\documentclass[review,onefignum,onetabnum]{siamonline190516}

\usepackage{lipsum}
\usepackage{amsfonts}
\usepackage{graphicx}
\usepackage{epstopdf}
\usepackage{algorithmic}
\ifpdf
  \DeclareGraphicsExtensions{.eps,.pdf,.png,.jpg}
\else
  \DeclareGraphicsExtensions{.eps}
\fi

\usepackage{enumitem}
\setlist[enumerate]{leftmargin=.5in}
\setlist[itemize]{leftmargin=.5in}

\newsiamremark{remark}{Remark}
\newsiamremark{hypothesis}{Hypothesis}
\crefname{hypothesis}{Hypothesis}{Hypotheses}
\newsiamthm{claim}{Claim}

\headers{Walking into the complex plane to `order' better time integrators}{J. D. George, S. Y, Jung, and N. M. Mangan}

\title{Walking into the complex plane to `order' better time integrators}

\author{Jithin D. George\thanks{Department of Engineering Sciences and Applied Mathematics, Northwestern University, Evanston, IL.\\
(\email{jithindgeorge93@gmail.com}, \email{samyjung@outlook.com}, \email{niallmm@gmail.com})}
\and Samuel Y. Jung \footnotemark[1]
\and Niall M. Mangan\footnotemark[1]}

\usepackage{amsopn}

\ifpdf
\hypersetup{
  pdftitle={Walking into the complex plane to `order' better time integrators.},
  pdfauthor={J. D. George, S. Y. Jung, and N. M. Mangan}
}
\fi

\usepackage{amsmath,mathtools}

\usepackage{tabularx}
\newenvironment{mat}{\left[ \begin{array}{ccccccccccccc}}{\end{array}\right]}
\newcommand\bcm{\begin{mat}}
\newcommand\ecm{\end{mat}}
\usepackage{amsfonts,epsfig}
\newcommand{\at}[2][]{#1|_{#2}}

\usepackage{forest}
\forestset{
  */.style={
    delay+={append={[]},}
  },
  rooted tree/.style={
    for tree={
      grow'=90,
      parent anchor=center,
      child anchor=center,
      s sep=2.5pt,
      if level=0{
        baseline
      }{},
      delay={
        if content={*}{
          content=,
          append={[]}
        }{}
      }
    },
    before typesetting nodes={
      for tree={
        circle,
        fill,
        minimum width=3pt,
        inner sep=0pt,
        child anchor=center,
      },
    },
    before computing xy={
      for tree={
        l=5pt,
      }
    }
  }
}
\DeclareDocumentCommand\rootedtree{o}{\Forest{rooted tree [#1]}}
\begin{document}
\maketitle

\begin{abstract}
Most numerical methods for time integration use real time steps. Complex time steps provide an additional degree of freedom, as we can select the magnitude of the step in both the real and imaginary directions.  By time stepping along specific paths in the complex plane, integrators can gain higher orders of accuracy or achieve expanded stability regions.  We show how to derive these paths for explicit and implicit methods, discuss computational costs and storage benefits, and demonstrate clear advantages for complex-valued systems like the Schrodinger equation. We also explore how complex time stepping also allows us to break the Runge-Kutta order barrier, enabling 5th order accuracy using only five function evaluations for real-valued differential equations.
\end{abstract}



\section{Introduction}

The real line is a tiny sliver in the whole of the complex plane. Wandering into the complex plane has often been employed to solve real-world problems. Two direct examples are the solution of complicated real-valued integrals via contour integrals \cite{brown2009complex} and the usage of the intuition gained from poles in the complex plane in control-theory applications \cite{nise2020control}. The complex plane provides wonderful insights to real problems, allowing mathematicians to impact practical applications.  One of the motivations for this paper was a simple numerical application  highlighted in a 2018 SIAM News article \cite{high18-diff} where Nick Higham describes how taking a numerical complex derivative neatly gets rid of round-off error. However, most methods for numerically solving a differential equation from time $a$ to time $b$, construct a path utilizing only the real line. We show that we can use the complex plane to our advantage when it comes to numerical time integration by taking time steps in the complex plane. 

This paper is not the first to talk about complex time steps. Early work on time stepping in the complex plane focused on systematically avoiding singularities during numerical integration of singular differential equations \cite{corliss1980integrating}. A relatively common use of complex time steps is in the field of operator splitting methods.  One of the earliest works in this field is \cite{chambers2003symplectic}, where Chambers explores operator splitting in symplectic integrators using complex substeps. Taking real substeps often results in negative substeps and substeps larger than the original step causing the symplectic integrator go back and forth along the direction of integration. By using complex substeps, Chambers showed how to avoid negative and unstable substeps, significantly reducing the truncation error. Chambers's work inspired others to consider complex time steps and complex coefficients as part of operator splitting methods for various parabolic and nonlinear evolution problems \cite{castella2009splitting,auzinger2017practical,casas2021compositions,blanes2010splitting, hansen2009high, casas2021high}. Complex coefficients have found uses in construction of numerical schemes for stiff differential algebraic systems \cite{al2006rosenbrock} and for Runge-Kutta-Nystrom methods \cite{gurkan2012fifth}. More recently, \cite{buvoli2019constructing} used Cauchy's integral formula as inspiration to introduce a general framework for deriving integrators using interpolating polynomials in the complex plane. This results in parallel Backward Differentiation Formula (BDF) methods with improved stability regions.

In this paper, we ask fundamental questions about the utility of numerical time-stepping in the complex plane, including: `What paths lead to the least error?' and `Are there general strategies and trade offs for complex time-stepping?'. Complex time steps have been previously studied in a similar manner to our work in \cite{filatov2006complex} where Filatov derives complex paths for linear differential equations where the Euler method gains higher orders of accuracy. In a similar direction, \cite{orendt2009geometry} uses linear analysis to show how taking a complex contour with specific roots of unity as time steps, allows Runge-Kutta methods to gain superconvergence, increasing the order of the method by one as the number of steps go to infinity. In our work, we go beyond the simple linear problem and demonstrate how to systematically derive higher-order paths with Runge-Kutta integrators for general nonlinear vector-valued differential equations.

We show that stepping into the complex plane provides an extra dimension that can be exploited to improve accuracy for general differential equations and design expanded stability regions (beyond what state of the art Runge-Kutta methods can do). We show how complex time integrators can circumvent the Runge-Kutta order barrier \cite{Butcher:2007, butcher2009fifth} and discuss the increased computational cost from complex operations.  The code to reproduce the numerical experiments in this paper is available on GitHub \footnote{\url{https://github.com/Dirivian/complex_time_integrators}} .

\section{Taking complex time steps with the Euler method}


  The idea of a complex time step can be unusual and unintuitive. To make it accessible, we first illustrate simple paths with complex time steps and demonstrate that applying the forward Euler method along the complex paths improves accuracy compared to forward Euler with real steps. The paths we choose in the complex plane must intersect the real line at the time points where we desire a solution. Our primary question is `what is the optimal n-step path in the complex plane over which to integrate between two real time points?'

To illustrate the simplest possible case, suppose we want to solve the equation $\dot{y} = \lambda y$ with the initial condition $y_0$ at time $t=0$. To find the solution $y(\Delta t)$, we can either take two real time steps of length $\frac{\Delta t}{2}$ or take a complex time step of length $\frac{\Delta t}{2} + i\frac{\Delta t}{2}$ followed by a complex time step of length $\frac{\Delta t}{2} - i\frac{\Delta t}{2}$ returning back to the real line. These two paths are shown in Fig. \ref{fig_1} (top). We numerically integrate the differential equations $\dot{y} = y$ along these paths as shown in Fig. \ref{fig_1} (bottom). By eye, the complex time path clearly does much better than the real time path in this linear case. The improvement with a complex path can hold for nonlinear ODEs, second order ODEs, non-autonomous ODEs (Fig. \ref{airyetall}) and nonlinear PDEs, such as the viscous Burgers equation (Fig. \ref{burgers}). In section \ref{higherordersection}, we describe the generalizations needed to solve nonlinear equations, but first we provide theoretical insight into the simple linear case.

\begin{figure}
\label{fig_1}
\begin{centering}
\includegraphics[width=5.3in]{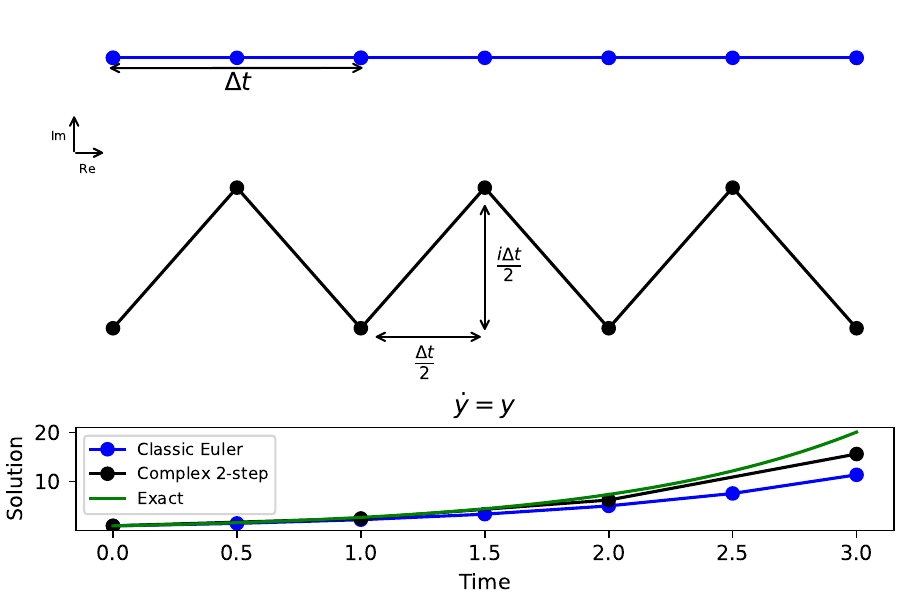}
\caption{Taking complex time steps that return to the real line results in increased accuracy.}
\end{centering}
\end{figure}

\begin{figure}
\label{airyetall}
\begin{centering}
\includegraphics[width=5.2in]{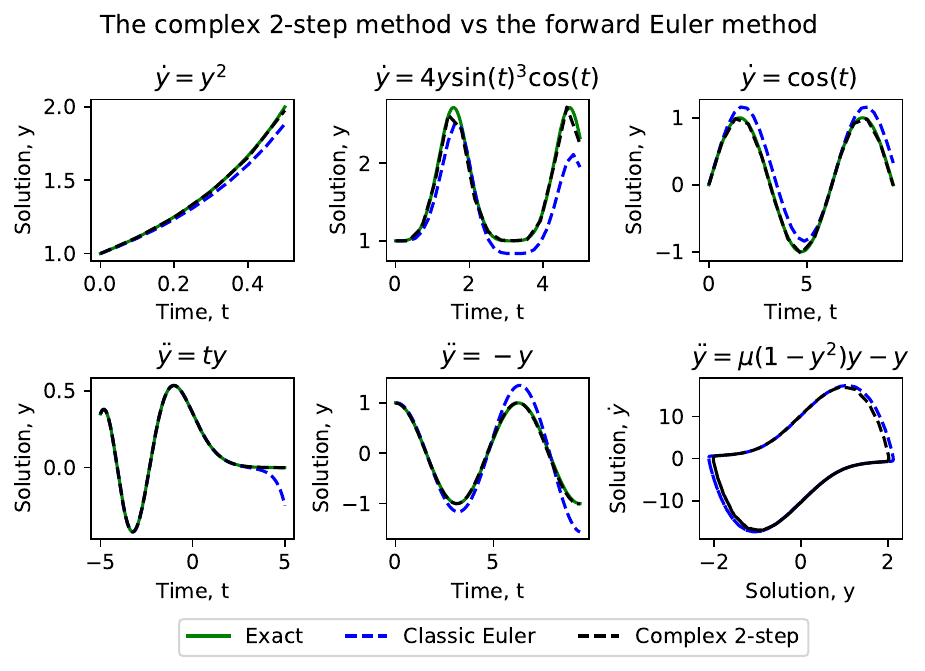}
\caption{Comparison of real-valued classic Euler method and 2-step complex-valued Euler method. The number of real timepoints for the classic scheme are roughly double, $2n-1$, the number for the 2-step complex scheme so that the net number of timesteps and function evaluations are the same for both schemes.}
\end{centering}
\end{figure}

\begin{figure}
\label{burgers}
\begin{centering}
\includegraphics[width=5.1in]{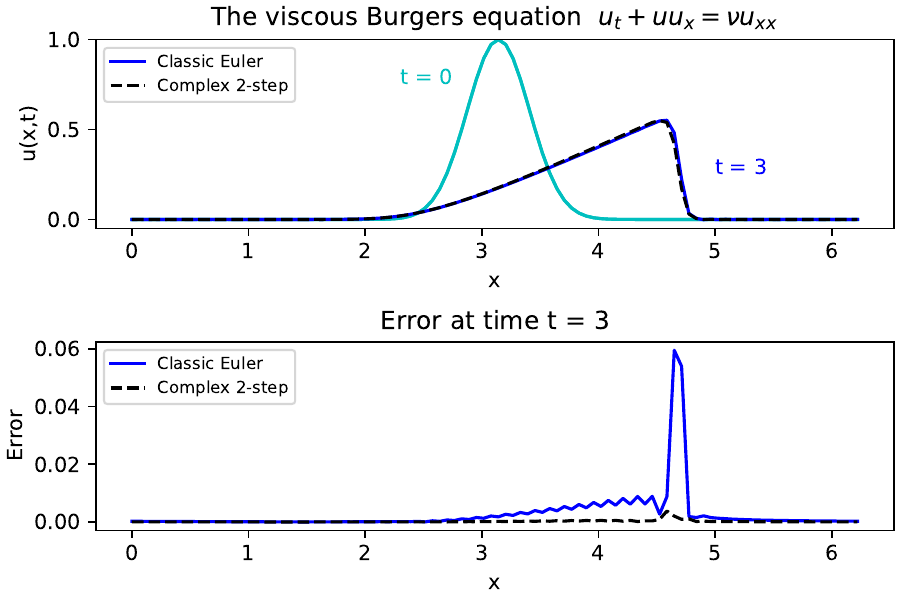}
\caption{The complex time steps reduce the error for the nonlinear viscous Burgers equation.}
\end{centering}
\end{figure}

Why does time stepping on the complex path improve accuracy of the numerical solution so significantly? To answer that, we analyze the linear problem $\dot{y} = y$ with initial condition $y(t_0) = y_0$. Consider time stepping from the initial condition $y_0$ at time $t_0$ to time $t+\Delta t$ using 2 real time steps. The discrete values of $y(t)$ are given by the following at $t = t_0, \; t_0+\frac{\Delta t}{2},$ and $t_0 + \Delta t$:

\begin{align}
    y(t_0) &= y_0 \nonumber\\
    y(t_0 + \frac{\Delta t}{2}) &= y_0 \bigg(1+\frac{\Delta t}{2} \bigg)\nonumber\\
     y(t_0 + \Delta t) &= y_0\bigg(1+\frac{\Delta t}{2}\bigg)\bigg(1+\frac{\Delta t}{2}\bigg) = y_0 \bigg(1 + \Delta t + \frac{\Delta t^2}{4}\bigg).
\end{align}

Now consider timestepping from $y_0$ at time $t_0$ to time $t+\Delta t$ using the two complex time steps.

\begin{align}
    y(t_0) &= y_0 \nonumber\\
    y\Big(t_0 + \frac{\Delta t}{2} +i\frac{\Delta t}{2}\Big) &= y_0 \bigg(1+\frac{\Delta t}{2}+i\frac{\Delta t}{2} \bigg)\nonumber\\
     y(t_0 + \Delta t) &= y_0\bigg(1+\frac{\Delta t}{2}+i\frac{\Delta t}{2}\bigg)\bigg(1+\frac{\Delta t}{2}-i\frac{\Delta t}{2}\bigg ) = y_0 \bigg(1 + \Delta t + \frac{\Delta t^2}{2}\bigg)
\end{align}

The exact solution to the problem is given by the Taylor expansion
\begin{align}
    y(t_0 +\Delta t) = y_0 e^{\Delta t} = y_0\Big(1+\Delta t + \frac{\Delta t^2}{2} + \hdots \Big).
\end{align}

The secret behind the higher accuracy of the complex 2-step method is that the polynomial obtained through the numerical approximation matches the Taylor series of the exact solution to the second order term, making forward Euler method gain second order accuracy. This is one of two unique 2-step paths in the complex plane which has this property. The other unique path uses the other permutation of the complex time steps $(\frac{\Delta t}{2} -i\frac{\Delta t}{2}, \frac{\Delta t}{2} +i\frac{\Delta t}{2})$ instead of $(\frac{\Delta t}{2} +i\frac{\Delta t}{2}, \frac{\Delta t}{2} -i\frac{\Delta t}{2})$. Both permutations of these steps will lead to the same polynomial with coefficients that match the Taylor expansion to 2nd order. This 2-step 2nd order complex Euler path originally appeared in \cite{filatov2006complex} where the author derives it as a higher-order path for linear differential equations. Later in Subsection \ref{subsec:nonlinearde}, we show how this path satisfies 2nd-order accuracy for general nonlinear differential equations.

By stepping into the complex plane, we gain an extra degree of freedom in the coefficients that allows us to design time-steps with improved accuracy. This extra dimension can be used to construct higher order paths for linear and nonlinear differential equations (Section \ref{higherordersection}), circumvent the order barrier for Runge-Kutta methods (Section \ref{sec:orderbarrier})  and even design improved stability regions (Section \ref{stabilitysection}). In the next section, we demonstrate how to extend the linear example to construct paths with higher order accuracy.

\section{Constructing a higher order time integrator} \label{higherordersection}

Is it possible then to construct 3-step path on which forward Euler would be 3rd order accurate? Would it be nth order accurate on an n-step complex path?
To build intuition for the simplest higher order construction, we continue analyzing the linear problem in Subsection \ref{subsec:linearFEpaths} and prove that achieving higher order requires complex time-stepping. We then  demonstrate how to construct implicit methods for linear problems in Subsection \ref{subsec:linearimplicitpaths}. Later in Subsections \ref{subsec:nonlinearde}, \ref{subsec:implicit_nonlinear} and  \ref{subsec:nonlineardeRK}, we extend the construction of complex time-steppers to the general nonlinear problem.

\subsection{Higher order integrators for linear differential equations using complex Euler steps} \label{subsec:linearFEpaths}

In this subsection, we show how to construct a third-order time integrator for the differential equation $\dot{y} = \lambda y$ by taking 3 forward Euler time steps of variable sizes in the complex plane that return to the real line at every desired time point.

If we take the three complex steps $w_1 \Delta t, w_2 \Delta t$ and $w_3 \Delta t$ to timestep from $y_0$ at $t_0$ to $t_0 +\Delta t$, the approximate numerical solution $\bar{y}$ at time $t +\Delta t$ is given by

\begin{align}
    \bar{y}(t_0 +\Delta t) = y_0(1+w_{1} \lambda \Delta t)(1+w_{2} \lambda \Delta t)(1+w_3\lambda \Delta t)  
\end{align}
where the $w_{i}$'s are complex variables. Collecting $\lambda \Delta t$ terms results in 
\begin{align}
\label{eq:linearcoeff}
\bar{y}(t_0 +\Delta t) 
&= y_0(1 + (w_1 +w_2 +w_3)\lambda \Delta t + (w_1w_2 + w_2w_3 + w_1w_3) \lambda^2 \Delta t^2  + w_1w_2w_3 \lambda^3 \Delta t^3).
\end{align}
The exact solution of the linear problem is given by 
\begin{align}
\label{eq:expcoeff}
y(t_0 +\Delta t)= y_0e^{\lambda \Delta t} = y_0(1 + \lambda \Delta t + \frac{1}{2} \lambda^2 \Delta t^2 +\frac{1}{6} \lambda^3 \Delta t^3 + \hdots )    
\end{align}

For the numerical approximation to be third-order accurate, the coefficients of $y_0(\lambda \Delta t)^n$ for $n=1,2,3$ in equations \eqref{eq:linearcoeff} and \eqref{eq:expcoeff} need to match. This results in the following order conditions:

\begin{align}
\label{nonlinearsystem1}
    w_1 + w_2 +w_3 &= 1 \nonumber\\
    w_1w_2 + w_2w_3 + w_1w_3 &= \frac{1}{2} \nonumber\\
    w_1w_2w_3 &= \frac{1}{6}.
\end{align}

One of the solutions to the above nonlinear system is 
\begin{align}
    (w_1, w_2, w_3) = (0.186731 + 0.480774 i, 0.626538, 0.186731 - 0.480774i).
\end{align}
However, as the nonlinear system \ref{nonlinearsystem1} is symmetric about $w_1, w_2$ and $w_3$, all 6 permutations of $(0.186731 + 0.480774 i, 0.626538, 0.186731 - 0.480774i)$ result in unique 3-step 3rd order complex time stepping paths. In Fig \ref{fig:pathslinear}, all six 3-step 3rd order paths are shown along with the two 2-step 2nd order paths and the single step forward Euler path.

For a general n-step nth order path, the order conditions would be

\begin{align}
    \sum_i w_i &= 1 \nonumber\\
   \sum_{i \neq j} w_iw_j &= \frac{1}{2} \nonumber\\
    \sum_{i \neq j \neq k} w_iw_j w_k &= \frac{1}{6} \nonumber\\
   & \vdots \nonumber\\
  \prod_{i=1}^{n}w_i &=\frac{1}{n!}.
\end{align}

Equivalent formulae for higher-order Euler paths for linear problems appear in \cite{filatov2006complex}, but in a less-compact fractional-step form. For higher orders, solving the nonlinear system of constraints on the coefficients becomes an increasingly tough symbolic task. However, solutions can be found numerically  up to a certain precision and be stored in a library for general use. 

 \begin{figure}[H]
\begin{centering}
\includegraphics[width=5.2in]{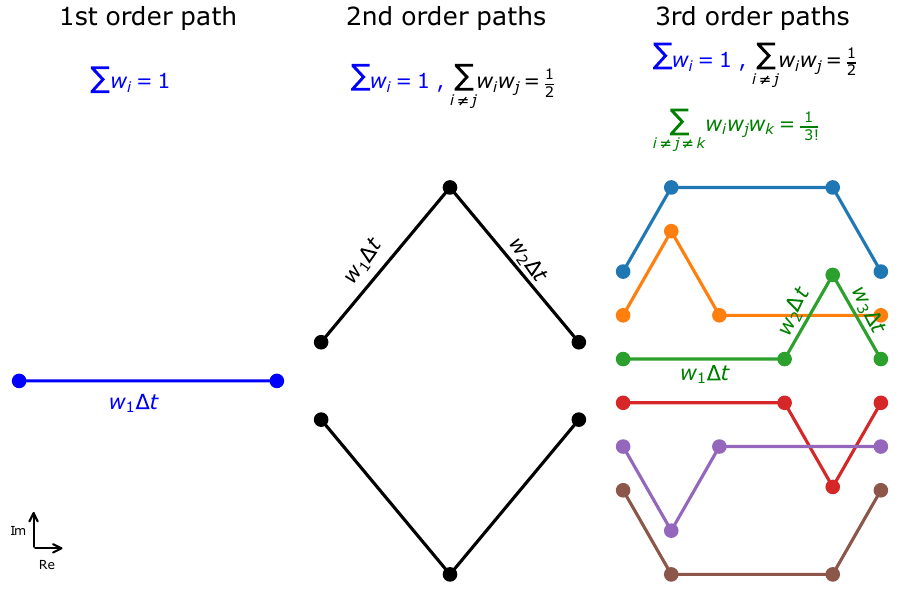}
\caption{All possible first, second, and third-order complex Euler methods for linear differential equations. Each order requires an additional step and adds an additional constraint on the coefficients.}
\end{centering}
\label{fig:pathslinear}
\end{figure}

These results beg the theoretical question: is there any way to achieve similar results with real-valued time-steps? Can one take steps of varying size on the real line and achieve a higher-order time integrator? We prove orders 2 or greater cannot be achieved with real-valued Euler steps.
\begin{theorem}

It is impossible to construct a time integrator with order greater than or equal to 2 using real-valued Euler time steps of variable sizes.
\end{theorem}

\begin{proof}
We prove the above theorem by showing that it is impossible to construct a second-order time integrator using forward Euler steps of variable size on the real line.

Suppose we use $n$ time steps of variable size given by $w_1, w_2, \hdots, w_n$. The approximate solution to $\dot{y} = \lambda y$ given by our time integrator $\bar{y}$ would be

\[\bar{y}(\Delta t) = (1+w_1 \lambda \Delta t)(1+w_2 \Delta \lambda t)\hdots(1+w_n \Delta \lambda t)  \]
where we assume the first order condition is satisfied, 
\[w_1 +w_2 + \hdots +w_n = 1.\]
Gathering together different order $\lambda \Delta t$ terms,
\begin{align*}
    \bar{y}(\Delta t) &= 1 + (w_1 +w_2 + \hdots +w_n)\lambda \Delta t +  (w_1w_2+w_1w_3 +\hdots w_{n-1}w_n) \lambda^2 \Delta t^2 +\hdots\\
     &= 1 + \lambda \Delta t +  (w_1w_2+w_1w_3 +\hdots w_{n-1}w_n) \lambda^2 \Delta t^2 +\hdots \\
\end{align*}
where we have substituted in our first order condition. For the time stepper to be second order, we need the coefficient on $\lambda^2 \Delta t^2$to be

\[w_1w_2+w_1w_3 +\hdots+ w_2w_3 + \hdots+ w_1w_n + w_2w_n + \hdots+ w_{n-1}w_n = \frac{1}{2}.\]
Factoring out a $w_n$ results in
\[w_1w_2+w_1w_3 +\hdots+ w_2w_3 + \hdots+  (w_1 +w_2 + \hdots +w_{n-1})w_n = \frac{1}{2}.\]
Solving the first order condition for $w_n$ and incorporating produces
\[w_1w_2+w_1w_3 +\hdots+ w_2w_3 + \hdots+  (w_1 +w_2 + \hdots +w_{n-1})(1- w_1 -w_2  \hdots -w_{n-1}) = \frac{1}{2},\]
which rearranges to
\[w_1 +w_2 + \hdots +w_{n-1} - w_1^2 - w_2^2\hdots -w_{n-1}^2- w_1w_2-w_1w_3 \hdots- w_{n-1}w_{n-2}= \frac{1}{2}.\]
Since all the $w$s are real, we can express them as 
\[w_j = k_j w_1 \quad j = 1, \hdots, n-1\]
where all the $k$s are real. Then, we get
\[(k_1 +k_2 + \hdots +k_{n-1})w_1 - w_1^2(k_1^2 + k_2^2\hdots +k_{n-1}^2+ k_1k_2+k_1k_3 \hdots+ k_{n-1}k_{n-2})= \frac{1}{2}.\]

Rearranging in terms of powers of $w_1$ one has
\[ w_1^2(k_1^2 + k_2^2\hdots +k_{n-1}^2+ k_1k_2+k_1k_3 \hdots+ k_{n-1}k_{n-2})-(k_1 +k_2 + \hdots +k_{n-1})w_1 + \frac{1}{2} = 0.\]

This is quadratic equation to solve for $w_1$ in the form of $ax^2+bx+c = 0$. For $w_1$ to only have real solutions, we need a positive discriminant $\Delta = b^2-4ac \geq 0 $. In  our case, the discriminant, $\Delta$, is

\begin{align*}
    \Delta &= (k_1 +k_2 + \hdots +k_{n-1})^2 - 4 (k_1^2 + k_2^2\hdots +k_{n-1}^2+ k_1k_2+k_1k_3 \hdots+ k_{n-1}k_{n-2}) \frac{1}{2} \nonumber\\ &= - (k_1^2 + k_2^2 + \hdots +k_{n-1}^2) 
\end{align*}
Since the discriminant is always negative, $w_1$ cannot be a purely real time step and must to be complex to satisfy our second order condition. Therefore, not only \textit{can} we achieve higher order integrators through complex time stepping, but we \textit{must} step into the complex plane to achieve higher order.

\end{proof}

\subsection{Higher order implicit integrators for linear differential equations using complex Euler steps} \label{subsec:linearimplicitpaths}

So far, we have managed to find 3rd order paths for time stepping linear ODEs and PDEs using explicit Euler steps. When dealing with stiff systems, it becomes necessary to use integrators that allow for large stable steps. Explicit methods usually have narrow stability regions. Hence, implicit methods are often preferred for stiff problems because they are often A-stable (their stability regions contain the entire left half plane) and L-stable (their solution goes to zero in a single step as the step size goes to infinity). We now demonstrate how to extend the analysis for linear problems to implicit methods by constructing a 2 step path using the implicit midpoint method that is 4th order accurate.

Consider the linear equation $\dot{y} = \lambda y$. Applying the implicit midpoint method to this equation produces the following difference equation,

\begin{align}
    y_{n+1} = y_n + \frac{\lambda \Delta t}{2} (y_n + y_{n+1})
\end{align}
which can also be written as
\begin{align}
    y_{n+1} = y_n\frac{1+\lambda \Delta t/2}{1-\lambda \Delta t/2}. 
\end{align}

Now, suppose we take two complex steps $w_1 \Delta t$ and $w_2 \Delta t$ to go from $y_n$ to $y_{n+1}$. Our numerical approximation at $y_{n+1}$ is

\begin{align}
\label{impcomplex}
    y_{n+1} = y_n\frac{(1+\lambda w_1 \Delta t/2)(1+\lambda w_2 \Delta t/2)}{(1-\lambda w_1 \Delta t/2)(1-\lambda w_2 \Delta t/2)}. 
\end{align}

If $y(t) = y_n$, the true solution at time $t+\Delta t$ would be $y_ne^{ \lambda \Delta t}$. This solution can be represented to $4$th order accuracy by the following rational function expansion of the exponential function, known as the (2,2) Pade approximant \cite{baker1996pade},

\begin{align}
\label{padeimp}
    y_{t + \Delta t} = y_n\frac{1+\lambda \Delta t/2 +\lambda^2 \Delta t^2/12}{1-\lambda \Delta t/2 +\lambda^2 \Delta t^2/12} + O(\Delta t^5).
\end{align}

Setting the right hand sides of \eqref{impcomplex} and \eqref{padeimp} equal to each other, gives us the 2-step path in the complex plane on which the implicit midpoint method has $4$th order accuracy.

By solving 
\begin{align}
    w_1 + w_2 = 1 \quad, \quad w_1w_2 = \frac{1}{3} 
\end{align}
we find that the two 2-step 4th order paths are $(\frac{1}{2} + i \frac{1}{2 \sqrt{3}}, \frac{1}{2} - i \frac{1}{2 \sqrt{3}})$ and  $(\frac{1}{2} - i \frac{1}{2 \sqrt{3}}, \frac{1}{2} + i \frac{1}{2 \sqrt{3}})$.

Equations \eqref{impcomplex} and \eqref{padeimp} highlight a clear connection between the optimal k-step complex path for the implicit midpoint method and the (k,k) Pade approximant.

In a similar way, we can derive the optimal k-step complex path for the Backward Euler method  by setting it equal to the (0,k) Pade approximate. If we take 3 complex step with the Backward Euler method before returning to the real line, the numerical solution is given by
\begin{align}
    y_{n+1} = \frac{y_n}{(1-w_1\lambda\Delta t)(1-w_2\lambda\Delta t)(1-w_3\lambda\Delta t)}.
\end{align}
The exact solution can be approximated to 3rd order accuracy by the (0,3) Pade approximant.
\begin{align}
\label{padeback}
    y_{t + \Delta t} = \frac{y_n}{1 - \lambda \Delta t +\lambda^2 \Delta t^2/2 +\lambda^3 \Delta t^3/6} + O(\Delta t^4).
\end{align}

This results in the same equations and paths as derived in Subsection \ref{subsec:linearFEpaths} for the forward Euler method.

We test out our newfound implicit paths on the heat equation with Dirichlet boundary conditions.
\begin{align}
    u_t = u_{xx}, \quad  x \in (0,1)
\end{align}

We choose the initial condition $u(x,0) = \sin(\pi x)$, which has the exact solution  $u(x,t) = e^{-\pi^2 t}\sin(\pi x)$. The spatial discretization is performed using 4th order finite differencing, periodic boundary conditions and a grid size of 10000 cells. As expected, the complex 3-step backward Euler method has 3rd order accuracy and the 2-step complex implicit midpoint method has 4th order accuracy (Fig. \ref{implicit_linear}).
      \begin{figure}
\begin{centering}
\includegraphics[width=5in]{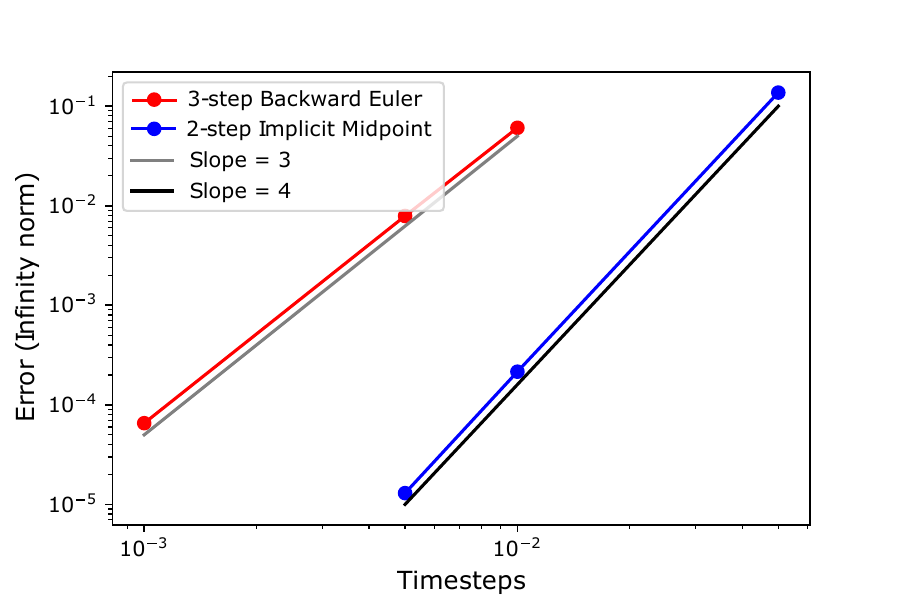}
\caption{The 2-step implicit midpoint method has 4th order accuracy while the 3-step Backward Euler has 3rd order accuracy for the heat equation.}
 \label{implicit_linear}
\end{centering}
\end{figure}

\subsection{Higher order explicit integrators for general nonlinear non-autonomous differential equations}\label{subsec:nonlinearde}

In the previous subsections, we showed how to construct a higher order complex path on which to time integrate using forward Euler, backward Euler and implicit midpoint steps. We experimentally demonstrated the expected order of accuracy for an example set of linear ordinary differential equations and partial differential equations. However, the derivation assumes a linear problem and does not guarantee the same order of accuracy for nonlinear problems. Indeed, when we tested the complex time-stepping schemes on nonlinear and non-autonomous equation, the 3 step forward Euler paths achieved only second-order accuracy. To achieve 3rd order accuracy for nonlinear, non-autonomous equations, we generalize the previously described procedure.

A general nonlinear differential equation is given by

\begin{align}
    \dot{y} = f(t,y), \quad y(t_0) = y_0.
\end{align}

Not every real-valued differential equation can be extended into in the complex plane. For the theoretical solution of a differential equation at a particular time to be the same along a path on the real line and  a path in the complex plane, the right hand side of the differential equation needs to be analytic over the domain enclosed by the two paths. Furthermore, some differential equations may only be defined over the real line. For example, for the differential equation $\dot{y} = \text{max}(1,y)$, the $\text{max}$ function is defined only for real arguments. For this reason, this paper only considers differential equations with natural extensions in the complex plane.
Let us also assume the function $f$ is sufficiently analytic near the region of integration i.e it does not have any singularities close to the path of integration. At time $t = t_0 + \Delta t$, we have

\begin{align}
\label{taylornonlinexp}
    y(t_0 + \Delta t) = y_0 + \Delta t f(t_0, y_0) + \frac{1}{2} \Delta t^2 \dot{f}(t_0,y_0) + \frac{1}{3!} \Delta t^3 \ddot{f}(t_0,y_0)  + \hdots
\end{align}

where
\begin{align}
    \dot{f}(t,y) = \frac{d}{dt}f(t,y)\at[\Bigg]{t=t_0, y = y_0} = f{\left(t_0,y_0 \right)} \frac{\partial}{\partial y} f{\left(t,y \right)}\at[\Bigg]{t=t_0, y = y_0} + \frac{\partial}{\partial t} f{\left(t,y \right)\at[\Bigg]{t=t_0, y = y_0}}.
\end{align}

Expanding equation \ref{taylornonlinexp}, we get

\begin{align}
\label{nonlintay}
    y(t_0+\Delta t) = &y_0 + \Delta t f(t_0, y_0) + \frac{1}{2} \Delta t^2 \bigg( f{\left(t,y \right)} \frac{\partial}{\partial y} f{\left(t,y \right)} + \frac{\partial}{\partial t} f{\left(t,y \right)}
\bigg)\at[\Bigg]{t=t_0, y = y_0}  \nonumber \\
    & + \frac{1}{3!} \Delta t^3 \bigg(f^2{\left(t,y \right)} \frac{\partial^{2}}{\partial y^{2}} f{\left(t,y \right)} + \left(\frac{\partial}{\partial y} f{\left(t,y \right)}\right)^{2} f{\left(t,y \right)}\nonumber \\ &+ 2f{\left(t,y \right)} \frac{\partial^{2}}{\partial y\partial t} f{\left(t,y \right)} + \frac{\partial}{\partial t} f{\left(t,y \right)} \frac{\partial}{\partial y} f{\left(t,y \right)} + \frac{\partial^{2}}{\partial t^{2}} f{\left(t,y \right)} \bigg)\at[\Bigg]{t=t_0, y = y_0}\\
    &+ \hdots  \nonumber 
\end{align}

Now, consider a complex time stepper that takes 3 complex time steps $w_1 \Delta t, w_2 \Delta t, w_3 \Delta t $ starting from $y_0$ to go to $y_1, y_2, y_3$. These steps are are given by:

\begin{align}
\label{nonlintay1}
    y_0 &= y(t_0)\\ 
    y_1 &= y_0 + w_1 \Delta t f( t_0, y_0)\\ 
        y_2 &= y_1 + w_2 \Delta t f( t_0 + w_1 \Delta t, y_1)\\ 
    y_3 &= y_2 + w_3 \Delta t f( t_0 + w_1 \Delta t+ w_2 \Delta t, y_2) \label{nonlintay4}.
\end{align}

One can expand $f(t_n, y_n)$ around $f(t_0,y_0)$ using the multivariate Taylor series. Doing so,

\begin{align}\label{3comptay}
    y_3 = &y_0 + \Delta t (w_{1}+w_{2}+w_{3}) f{\left(t_0,y_0 \right)} \nonumber\\&+  \Delta t^{2} (w_1 w_2 + w_1 w_3 +w_2 w_3)\bigg(f{\left(t,y \right)} \frac{\partial}{\partial y} f{\left(t,y \right)} +  \frac{\partial}{\partial t} f{\left(t,y \right)}  \bigg)\at[\Bigg]{t=t_0, y = y_0} \nonumber\\&
    +  \Delta t^{3} \bigg( w_{1} w_{2} w_{3} \bigg(f{\left(t,y \right)} \left(\frac{\partial}{\partial y} f{\left(t,y \right)}\right)^{2} + \frac{\partial}{\partial t} f{\left(t,y \right)} \frac{\partial}{\partial y} f{\left(t,y \right)}+ \frac{\partial}{\partial t} f{\left(t,y \right)} \frac{\partial}{\partial y} f{\left(t,y \right)}\bigg) \nonumber\\& + (w_{1}^{2} w_{2}+ w_{1}^{2} w_{3}+ 2w_1 w_2 w_3 + w_{2}^{2} w_{3}) \bigg(\frac{1}{2} f^{2}{\left(t,y \right)} \frac{\partial^{2}}{\partial y^{2}} f{\left(t,y \right)}  \nonumber\\& + \frac{1}{2}\frac{\partial^{2}}{\partial t^{2}} f{\left(t,y \right)}+f{\left(t,y \right)} \frac{\partial^{2}}{\partial y\partial t} f{\left(t,y \right)}  \bigg)\nonumber \bigg)\at[\Bigg]{t=t_0, y = y_0} + O\left(\Delta t^{4}\right).
\end{align}

For this three step complex integrator to have 3rd order accuracy, the coefficients of the partial derivatives in  \ref{3comptay} should match that of equation \ref{nonlintay}. Equating coefficients, we get this system of nonlinear equations.

\begin{align}
    \label{syseq}
    w_1 +w_2 +w_3 = 1 \nonumber\\
    w_1w_2 + w_1w_3 + w_2w_3 = \frac{1}{2} \nonumber\\
       w_1 w_2 w_3 = \frac{1}{6} \nonumber\\
        w_{1}^{2} w_{2} +  w_{1}^{2} w_{3}  + 2 w_{1} w_{2} w_{3}  +  w_{2}^{2} w_{3}  = \frac{1}{3}
\end{align}

The first three equations are the same as those derived for the linear system equation \ref{nonlinearsystem1}, and the 4th equation arises from the nonlinear terms in the multivariate expansion. Unfortunately, this system is overconstrained and does not have a solution, implying that we need more than three steps for 3rd order accuracy. However, for systems with only real solutions, we can take further advantage of the degree of freedom in the imaginary part of the step and find many 3-step paths with 3rd order accuracy. If we allow the  truncation error beyond $O(\Delta t)$ to be purely imaginary, we can remove this imaginary error by taking only the real part of the numerical solution at the end of the 3-step sequence.
In order to find a 3-step 3rd order path for real-valued differential equations, let us allow the truncation error at the order of $\Delta t^3$ to be purely imaginary,  Let $y$ be the true solution and $\bar{y}$ be the approximate numerical solution. Then, with our setup, we would have

\begin{align}
    y(t+\Delta t) = \bar{y}(t+\Delta t)+ ik\Delta t^3  + O(\Delta t^4)
\end{align}
where $k$ is real.
Taking the real part on both sides, we get
\begin{align}
    \textbf{Re}(y(t+\Delta t)) = \textbf{Re}(\bar{y}(t+\Delta t))  + O(\Delta t^4).
\end{align}
If the true solution to our differential equation and its initial condition are known to be real-valued at real time points, then we have $y(t+\Delta t) =  \textbf{Re}(y(t+\Delta t))$ and taking the real part of the solution at the end of complex path results in 3rd order accurate numerical solution. This trick of enforcing purely imaginary error at $O(\Delta t^3)$ produces a system with fewer constraints on the coefficients. The coefficients must satisfy the following relaxed conditions:

\begin{align}
    \label{syseq2}
    w_1 +w_2 +w_3 = 1 \nonumber\\
    w_1w_2 + w_1w_3 + w_2w_3 = \frac{1}{2} \nonumber\\
       \textbf{Re}(w_1 w_2 w_3) = \frac{1}{6} \nonumber\\
        \textbf{Re}(w_{1}^{2} w_{2} +  w_{1}^{2} w_{3}  + 2 w_{1} w_{2} w_{3}  +  w_{2}^{2} w_{3})  = \frac{1}{3}.
\end{align}

There are many paths satisfying these equations. Note that we could also have found such  3rd order paths by allowing the error at $O(\Delta t^2)$ to be purely imaginary.   Of the 3rd order paths in Fig \ref{fig:pathslinear} that were found to satisfy the linear order conditions, the top and bottom (cyan and brown) paths also satisfy the 3rd order nonlinear order conditions given in \ref{syseq2}. We have tested our complex time-stepping schemes on a variety of differential equations to show that they satisfy the order conditions. The equations are listed in Table \ref{DElist} along with a unique label corresponding to their legend on the convergence plot in Fig.\ref{convergence_all}. 
  \begin{figure}[H]
\begin{centering}
\includegraphics[width=6in]{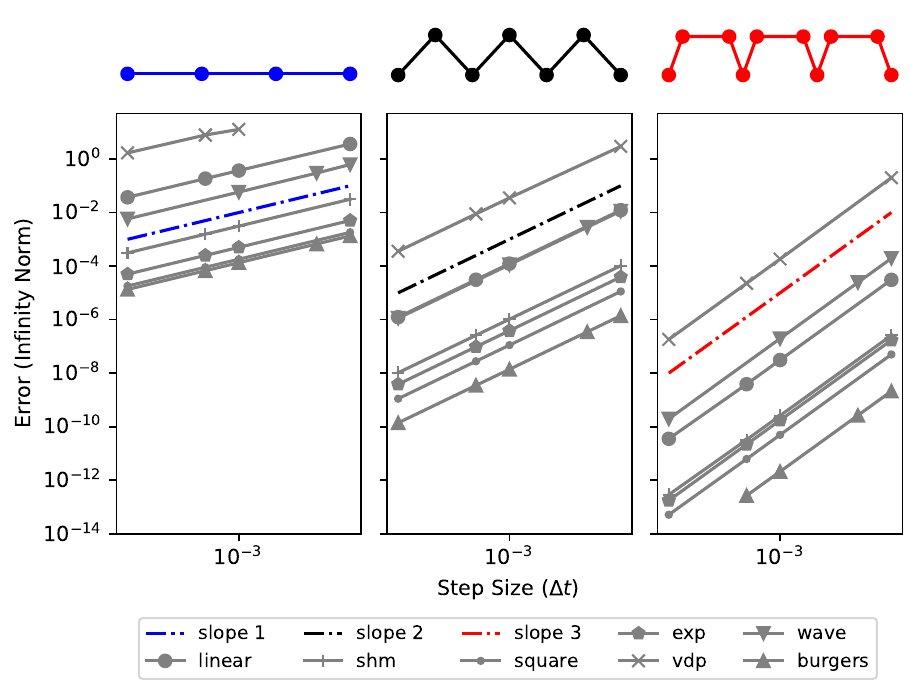}
\caption{Convergence of various ODEs and PDEs (Sec. \ref{subsec:nonlinearde}) for 1-step (forward Euler), 2-step and 3-step complex time steps. The path/method used is displayed above the corresponding convergence plot. The orders of accuracy are as expected.}
 \label{convergence_all}
\end{centering}

\end{figure}

\begin{table}
\label{DElist}
\caption{List of differential equations}
\begin{tabularx}{\textwidth}{|>{\setlength\hsize{\hsize}\setlength\linewidth{\hsize}}X|}
\hline
1. Linear ODEs
\begin{itemize}
\item The Dahlquist test problem (`linear')
\begin{align*}
    \dot{y} = \lambda y,\quad y(0) = 1
\end{align*}
\item  Simple harmonic oscillator (`shm')
    \begin{align*}
    \ddot{y} = -y, \quad y(0) = 1, \quad y'(0) = 0 
\end{align*}
\end{itemize}
\vphantom{1. Real unequal eigenvalues of same sign}
\\
\hline
2. Nonlinear ODEs
\begin{itemize}
    \item Nonlinear autonomous differential equation 1(`square')
\begin{align*}
    \dot{y} = - y^2, \quad y(0) = 1
\end{align*}

    \item Nonlinear autonomous differential equation 2 (`exp')
\begin{align*}
    \dot{y} = -e^y, \quad y(0) = 1
\end{align*}

    \item Nonlinear sine(`nlsin')
\begin{align*}
    \dot{y} = 4y\sin{(t)}^3\cos{(t)}, \quad y(0) = 1
\end{align*}
\item  The Van Der Pol Oscillator (`vdp')
\begin{align*}
    \dot{y_1} = y_2, \dot{y_2} = \mu(1-y_1^2)y_2-y_1, \quad (y_1, y_2)(0) = (2,0)
\end{align*}
\end{itemize}
\vphantom{2. Real unequal eigenvalues of opposite sign}
\\
\hline
3. Linear PDE : The advection equation (`wave')
 
 This is a linear PDE given by 
 \begin{align}
     u_{t} = cu_x, \quad x \in (0,2\pi), \quad t \in (0, 1)
 \end{align}
 
 This PDE is solved with periodic boundary conditions, the initial condition $u(x, t=0) = e^{7(x-\pi)^2}$, and $c = 1$. The exact solution is  $u(x, t) = e^{7(x-t-\pi)^2}$. The Fourier spectral method is used for the spatial differencing with 70 modes. 
\\
\hline
3. Nonlinear PDE : Viscous Burgers equation (`burgers')

  This is a nonlinear PDE given by 
 \begin{align}
     u_{t} +uu_{x} = \nu u_{xx}, \quad x \in (0,2\pi), \quad t \in (0,2)
 \end{align}
 
 This PDE is solved with periodic boundary conditions, the initial condition $u(x, t=0) = 2 \mu \sin x/(1.5+\cos x)
$, and $nu = 0.1$. The exact solution is  $u(x, t) = 2 \mu e^{-\mu t}\sin x/(1.5+e^{-\mu t}\cos x)
$. The Fourier spectral method is used for the spatial differencing with 70 modes. 
\\
\hline
\end{tabularx}
\end{table}

 Apart from the Van Der Pol oscillator, the error for the ODEs and PDEs is calculated by comparing the numerical solution to the exact solution. For the Van Der Pol oscillator, the numerical solution using complex paths is compared to the numerical solution obtained by using the `RK44' 4th order method from Nodepy \cite{ketcheson2020nodepy} with $\Delta t = 10^{-6}$.

 \subsection{Higher order implicit integrators for general nonlinear non-autonomous differential equations} \label{subsec:implicit_nonlinear}
 
 We now explore whether similar results can be achieved for real-valued systems using implicit methods. Previously, we found that the path $(\frac{1}{2} + i \frac{1}{2 \sqrt{3}}, \frac{1}{2} - i \frac{1}{2 \sqrt{3}})$  allows the implicit midpoint method to be 4th order accurate when solving linear differential equations. Does the expected 4th order accuracy of the 2-step complex path for the implicit midpoint method derived in the previous subsection hold for general nonlinear differential equations?  Numerically, we confirm this to be true, provided we take only the real part of the numerically calculated values at the end of each 2-step complex sequence  (Fig. \ref{implicit_nonlinear}). To demonstrate 4th order accuracy analytically, we analyze the midpoint method following reasoning similar to our previous nonlinear analysis for explicit methods (Equations  \eqref{nonlintay1} - \eqref{nonlintay4}). Suppose that we take two implicit steps $w_1 \Delta t$ and  $w_2 \Delta t$ from $y_n$ at $t_n$ to $y_{n+1}$ at at $t + \Delta t $ with $y_{m}$  being the intermediate solution at at $t+ w_1 \Delta t$.

\begin{align}
    \label{impnon1}
    y_n &= y(t_n)\\
    \label{impnon2}
    y_m &= y_n + w_1 \Delta t f(t_n+w_1 \Delta t /2,  (y_n +y_m)/2) \\
    \label{impnon3}
    y_{n+1}  &= y_m + w_2 \Delta t f(t_n+w_1 \Delta t+ w_2 \Delta t /2,  (y_m +y_{n+1})/2)
\end{align}

Unlike the previous analysis, equations \eqref{impnon2} and \eqref{impnon3} are implicit equations and cannot be directly Taylor expanded around $f(t,y_n)$. To overcome this difficulty and calculate the desired $w_1$ and $w_2$ step sizes, we approximate the implicit equations with explicit versions that are accurate to at least 4th order. The explicit approximations are found by recursively substituting the implicit equation into itself in an iterative process similar to Picard Iteration. For example, if you have an 
implicit equation given by
\begin{align} \label{simpleimplicit}
    y = y_0 + \Delta t g(y)
\end{align}
one can recursively apply it twice to obtain the following explicit equation, 

\begin{align}
    y = y_0 + \Delta t g(y_0 + \Delta t g(y_0 + \Delta t g(y_0)))
\end{align}

which is accurate to Eq. \ref{simpleimplicit} to the 3rd order. To obtain the 4th order accurate explicit approximation of equations \eqref{impnon2} and \eqref{impnon3}, they need to be recursively applied at least three times. Note that the explicit equations are 4th order accurate approximations of the original implicit difference equations and not the original differential equation.

At the end of this process,, we substitute $(w_1, w_2) = (\frac{1}{2} + i \frac{1}{2 \sqrt{3}}, \frac{1}{2} - i \frac{1}{2 \sqrt{3}})$ into the explicit approximation of equation \eqref{impnon3} and obtain 

\begin{align}
\label{implicittay}
    y_{n+1} = &y_0 + \Delta t f(t_n, y_n) + \frac{1}{2} \Delta t^2 \bigg( f{\left(t,y \right)} \frac{\partial}{\partial y} f{\left(t,y \right)} + \frac{\partial}{\partial t} f{\left(t,y \right)}
\bigg)\at[\Bigg]{t=t_n, y = y_n}  \nonumber \\
    & + \frac{1}{3!} \Delta t^3 \bigg(f^2{\left(t,y \right)} \frac{\partial^{2}}{\partial y^{2}} f{\left(t,y \right)} + \left(\frac{\partial}{\partial y} f{\left(t,y \right)}\right)^{2} f{\left(t,y \right)}\nonumber \\ &+ 2f{\left(t,y \right)} \frac{\partial^{2}}{\partial y\partial t} f{\left(t,y \right)} + \frac{\partial}{\partial t} f{\left(t,y \right)} \frac{\partial}{\partial y} f{\left(t,y \right)} + \frac{\partial^{2}}{\partial t^{2}} f{\left(t,y \right)} \bigg)\at[\Bigg]{t=t_n, y = y_n} \nonumber\\
    &+ \frac{1}{4!}dt^{4} \bigg( (1+0.096225i)f^{3}{\left(t,y \right)} \frac{\partial^{3}}{\partial y^{3}} f{\left(t,y \right)}\nonumber\\& + (4-0.288675i) f^{2}{\left(t,y \right)} \frac{\partial}{\partial y} f{\left(t,y \right)} \frac{\partial^{2}}{\partial y^{2}} f{\left(t,y \right)}  + (3+0.288675i) f^{2}{\left(t,y \right)} \frac{\partial^{3}}{\partial y^{2}\partial t} f{\left(t,y \right)}\nonumber\\& + 3 f{\left(t,y \right)} \frac{\partial}{\partial t} f{\left(t,y \right)} \frac{\partial^{2}}{\partial y^{2}} f{\left(t,y \right)} + f{\left(t,y \right)} \left(\frac{\partial}{\partial y} f{\left(t,y \right)}\right)^{3} \nonumber\\ & + (5-0.57735i) f{\left(t,y \right)} \frac{\partial}{\partial y} f{\left(t,y \right)} \frac{\partial^{2}}{\partial y\partial t} f{\left(t,y \right)}  +  (3+0.288675i)  f{\left(t,y \right)} \frac{\partial^{3}}{\partial y\partial t^{2}} f{\left(t,y \right)} \nonumber\\ & + \frac{\partial}{\partial t} f{\left(t,y \right)} \left(\frac{\partial}{\partial y} f{\left(t,y \right)}\right)^{2} + 3 \frac{\partial}{\partial t} f{\left(t,y \right)} \frac{\partial^{2}}{\partial y\partial t} f{\left(t,y \right)}  +(1-0.096225i) \frac{\partial^{3}}{\partial t^{3}} f{\left(t,y \right)}  \nonumber\\ &+ (1-0.288675i)\frac{\partial^{2}}{\partial t^{2}} f{\left(t,y \right)} \frac{\partial}{\partial y} f{\left(t,y \right)}
    \bigg)\at[\Bigg]{t=t_n, y = y_n} + O(\Delta t^5). 
\end{align}

If we  compare \eqref{implicittay} to the first 5 terms of Taylor expansion in equation \eqref{nonlintay}, we see that the error at O($\Delta t^4$) is purely imaginary, which demonstrates we can get 4th order accuracy for real-valued systems by extracting the real component of the solution.

Now, that we have proved our two-step implicit midpoint method to be 4th order accurate for real-valued nonlinear systems, we test it on the Van Der Pol oscillator,

\begin{align}
    \dot{y} = x - \frac{x^3}{3} -\frac{\dot{x}}{\mu}
\end{align}
which can be rewritten as the following system

\begin{align}
    \dot{x} &= y \nonumber\\
    \dot{y} &= \mu (1-x^2)y-x.
\end{align}

We set $\mu = 10$ and the initial condition as $(x,y) = (2,0)$. We achieve the expected 4rd order accuracy for the 2-step complex implicit method (Fig. \ref{implicit_linear}).

      \begin{figure}
\begin{centering}
\includegraphics[width=5in]{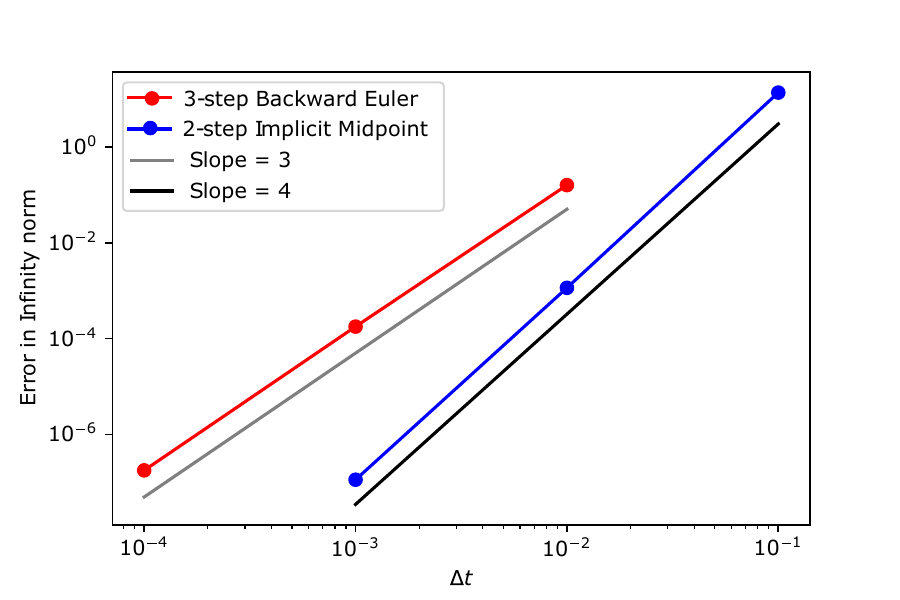}
\caption{The 2-step implicit midpoint method has 4th order accuracy while the 3-step Backward Euler has 3rd order accuracy for the Van Der Pol oscillator.}
 \label{implicit_nonlinear}
\end{centering}
\end{figure}
 
\subsection{Complex paths and coefficients for general Runge-Kutta methods}\label{subsec:nonlineardeRK}

So far, we have derived paths in the complex plane that satisfy the order conditions for scalar non-autonomous differential equations. We also applied these paths to vector-valued non-autonomous equations including systems and partial differential equations, and demonstrated up to 3rd order accuracy with the numerical results in subsection \ref{subsec:nonlinearde}.  However, Butcher showed in \cite{butcher2009fifth} that the order conditions for scalar non-autonomous differential equations are the same as those for vector-valued differential equations only up to order 4. Thus, the approach used so far is insufficient to derive the correct order conditions for a fifth order or higher complex path for vector-valued differential equations. In order to obtain the order conditions for general time-steppers for vector-valued equations, we need to utilize the theory of Butcher series (B-Series) and rooted trees \cite{butcher1972algebraic, hairer1974butcher}. B-Series allow us to represent both the exact solutions to general differential equations and their  numerical solutions created from a broad class of integrators,  as a series expansion in $\Delta t$ similar to the Taylor series analysis we have been doing so far. Runge-Kutta methods and methods created by taking complex time steps with Runge-Kutta methods fall into the class of integrators whose numerical solutions can be represented by B-Series. In this section, we describe the process of deriving order conditions for complex time-steppers using B-Series and show how the software BSeries.jl \cite{ketcheson2021computing} simplifies that process tremendously. Finally, as an example, we re-derive the order conditions necessary for the 3-step 3rd order complex Euler method using  and show that it is the same as that obtained by our Taylor series analysis in  \ref{syseq}.

Consider the first order system of differential equations given by

\begin{align}
    \dot{y} = f(y), \quad y(t_0) = y_0.
\end{align}

The exact solution at time $t = t_0 + \Delta t$ is given by 
\begin{align}
    y(t_0 + \Delta t) = y_0 + \Delta t \dot{y} (t_0)+ \frac{1}{2} \Delta t^2 \ddot{y} (t_0)+ \frac{1}{3!} \Delta t^3 \dddot{y} (t_0)+ \hdots
\end{align}

Each term in the above expansion can be written as the combination of products of partial derivatives of $f$, known as elementary differentials. Using $f'$ as short-hand for $\frac{\partial f}{\partial t}$, we can rewrite each term in the expansion in terms of their elementary differentials.

\begin{align}
    \label{elementarydiff}
    \dot{y} &= f \nonumber\\
     \ddot{y} &= f'(f) \nonumber\\
       \dddot{y} &= f''(f,f) + f'(f'(f))
\end{align}

The set of elementary differentials is isomorphic to the set of rooted trees and thus, each elementary differential has an associated rooted tree. We can rewrite equations \ref{elementarydiff} in terms of the associated trees.

\begin{align}
    \label{elementarydiff_rooted}
    \dot{y} &= F_{f}\mathopen{}\left( \rootedtree[] \right)\mathclose{} \nonumber\\
     \ddot{y} &= F_{f}\mathopen{}\left( \rootedtree[[]] \right)\mathclose{} \nonumber\\
       \dddot{y} &= F_{f}\mathopen{}\left( \rootedtree[[][]] \right)\mathclose{} + F_{f}\mathopen{}\left( \rootedtree[[[]]] \right)\mathclose{}
\end{align}
where $F_f(\tau)$ is the elementary differential associated with the rooted tree $\tau$.

We can also write the  exact solution at time $t = t_0 + \Delta t$ in terms of the associated rooted trees, creating a B-Series for the exact solution.

\begin{align}
\label{eqn:bseriesExact}
    y(t_0 + \Delta t) = y_0 + \sum_{\tau \in T} \frac{\Delta t^{|\tau|}}{ \sigma(\tau) \gamma(\tau)}F_f(\tau) (y) (t_0) 
\end{align}

where $F_f(\tau)$ is the elementary differential associated with the rooted tree $\tau$, $T$ is the set of all rooted trees, $|\tau|$ is the number of nodes in the rooted tree $\tau$, $\sigma(\tau)$ is the symmetry of the rooted tree $\tau$ and $\gamma(\tau)$ is the density of the rooted tree $\tau$. The rooted trees, elementary differentials, symmetries and densities can all be defined recursively and a more elaborate discussion of them can be found in \cite{mclachlan2015butcher, butcher2021b, ketcheson2021computing,  butcher1972algebraic, hairer1974butcher}.

Numerical time stepping methods that have an associated B-Series, such as Runge-Kutta methods, are called B-Series methods \cite{butcher1972algebraic, hairer1974butcher}. For a Runge-Kutta method with elementary weights $\Phi$, its solution at time $t = t_0 + \Delta t $ can be written in terms of its B-Series.

\begin{align}
\label{eqn:bseriesRK}
    y(t_0 + \Delta t) = y_0 + \sum_{\tau \in T} \frac{\Delta t^{|\tau|}}{ \sigma(\tau)} \Phi(\tau) F_f(\tau) (y) (t_0) 
\end{align}

One can now derive the order conditions needed to be satisfied by the elementary weights by matching the elementary differentials in equation \ref{eqn:bseriesRK}  with the elementary differentials in equation \ref{eqn:bseriesExact} up to the desired order.

The complex time steppers that we have mentioned in this paper also have associated B-Series. We can derive their B-Series using the composition of the B-Series associated with each individual Euler steps. Once we construct the B-Series, we can derive the order conditions needed to be satisfied by the complex steps by matching the B-Series to the B-Series of the exact solution up to the desired order. We can do this same process for any complex time integrator obtained by taking complex time steps with any Runge-Kutta method. We can also let all the coefficients of the Runge-Kutta method be complex valued and further increase our degrees of freedom.

Ketcheson and Ranocha  recently released a package in Julia named BSeries.jl \cite{ketcheson2021computing} that allows for the symbolic construction, composition and manipulation of B-Series. We now demonstrate how BSeries.jl can be used to construct B-Series for a general complex time stepper and obtain order conditions.

Runge-Kutta methods can be completely defined by their coefficients $A,b$ and $c$. The following code constructs the B-Series for a Forward Euler method taking a time step of $w_1 \Delta t$ from its coefficients $A,b$ and $c$. The coefficient $b$ which is $1$ for a time step of $\Delta t$, is modified to $w_1$  for the time step of $w_1\Delta t$. An alternate way of obtaining the B-Series for the numerical solution after a time step of $w_1\Delta t$ is to replace every instance of $\Delta t$ with $w_1 \Delta t$ in the B-Series corresponding to a time step of $\Delta t$.

\begin{verbatim}
        w_1 = symbols("w_1", real=true)
        A =  [0;]
        A = A[:,:]
        b =  [w_1]
        c = [0]
        coefficients_w1 = bseries(A, b, c, 5)
        latexify(coefficients_w_1, cdot=false)
\end{verbatim}

This gives us the following B-Series.
\begin{align}
   y(t+w_1\Delta t) =  y_0 + \Delta t w_{1} F_{f}\mathopen{}\left( \rootedtree[] \right)\mathclose{} 
\end{align} 

Similarly, we can find B-Series for the numerical solution from the Forward Euler method after a time step of $w_2 \Delta t$ and a time step of $w_3 \Delta t$. We can then compose the B-Series together to get the B-Series for the numerical solution after  taking three complex Euler time steps.

We can run the following code to find the B-Series for the Forward Euler method after two time steps of size $w_1 \Delta t$ and $w_2 \Delta t$

\begin{verbatim}
        coeff_w1w2 = compose(coefficients_w1,coefficients_w2);
        latexify(coeff_w1w2, cdot=false)
\end{verbatim}

\begin{align}
\label{eqn:bseriesjlw1w2}
   y(t+w_1\Delta t+w_2\Delta t) =  y_0  + \Delta t \left( w_{1} + w_{2} \right) F_{f}\mathopen{}\left( \rootedtree[] \right)\mathclose{} + \Delta t^{2} w_{1} w_{2} F_{f}\mathopen{}\left( \rootedtree[[]] \right)\mathclose{} + \hdots
\end{align} 

Similarly, we can find the B-Series for the Forward Euler method after the time steps $w_1 \Delta t$ , $w_2 \Delta t$ and $w_3 \Delta t$.

\begin{verbatim}
        coeff_w1w2w3 = compose(coeff_w1w2,coefficients_w3);
        latexify(coeff_w1w2w3, cdot=false)
\end{verbatim}

\begin{align}
\label{eqn:bseriesjlw1w2w3}
   y(t+(w_1+w_2+w_3)\Delta t) = & y_0+ \Delta t \left( w_{1} + w_{2} + w_{3} \right) F_{f}\mathopen{}\left( \rootedtree[] \right)\mathclose{} + \Delta t^{2} \left( w_{1} w_{2} + w_{3} \left( w_{1} + w_{2} \right) \right) F_{f}\mathopen{}\left( \rootedtree[[]] \right)\mathclose{}\nonumber  \\& + \Delta t^{3} w_{1} w_{2} w_{3} F_{f}\mathopen{}\left( \rootedtree[[[]]] \right)\mathclose{} + \Delta t^{3} \left( \frac{w_{1}^{2} w_{2}}{2} + \frac{w_{3} \left( w_{1} + w_{2} \right)^{2}}{2} \right) F_{f}\mathopen{}\left( \rootedtree[[][]] \right)\mathclose{} +\hdots
\end{align} 

You can also find the B-Series of the exact solution with a single line.

\begin{verbatim}
        latexify(ExactSolution(coeff_w1w2w3), cdot = false)
\end{verbatim}

\begin{align}
    \label{eqn:bseriesjlexact}
     y(t+\Delta t) = & y_0+  \Delta t F_{f}\mathopen{}\left( \rootedtree[] \right)\mathclose{} + \frac{\Delta t^{2}}{2} F_{f}\mathopen{}\left( \rootedtree[[]] \right)\mathclose{} + \frac{\Delta t^{3}}{6} F_{f}\mathopen{}\left( \rootedtree[[[]]] \right)\mathclose{} + \frac{\Delta t^{3}}{6} F_{f}\mathopen{}\left( \rootedtree[[][]] \right)\mathclose{}
\end{align}

Matching the coefficients in equations \ref{eqn:bseriesjlexact} and \ref{eqn:bseriesjlw1w2} gives us the conditions necessary to obtain 2nd order accuracy. Matching the coefficients in equations \ref{eqn:bseriesjlexact} and \ref{eqn:bseriesjlw1w2w3} gives us the conditions necessary to obtain 3rd order accuracy and we can see that the order conditions are exactly those we previously derived in equation \ref{syseq}.

The theory of B-Series and the symbolic package BSeries.jl  thus provides us with a systematic way to find the higher order complex paths for general Runge-Kutta methods without much difficulty.

\section{Computational cost of complex timestepping schemes} \label{section:cost}

 In this section, we discuss a major drawback of the complex time stepping schemes, the increased computational cost for performing complex operations.
 
 Consider the forward Euler method applied to the differential equation $\dot{y} = f(t,y)$, yielding the iteration scheme $ y_{i+1} = y_i + \Delta t f(t,y)$. If the differential equation is linear, the most expensive part of the scheme is the multiplication $\Delta t f(t,y) = \Delta t \lambda y$.  Naively multiplying two complex numbers costs roughly four times as much as multiplying two real numbers. If we split the complex difference equation into its real and imaginary parts, we get the following systems made up of purely real variables and  we can clearly see that the number of multiplications needed at each step has increased by a factor of 4. 
\begin{align} \label{cost}
    \bcm y_r  \\y_i\ecm^{k+1} = \bcm y_r  \\y_i\ecm^{k} + \bcm \lambda_r \Delta t_r & -\lambda_i \Delta t_i \\ \lambda_i \Delta t_i & \lambda_r \Delta t_r\ecm \bcm y_r  \\y_i\ecm^{k}
\end{align}

So, we expect the cost of solving linear differential equations using complex time integration could be as high as 4 times that of their real equivalents. Indeed, when solving the real-valued equation $\dot{y} =y$, a 2-step complex Euler integrator has a much larger running time than the real-valued mid-point method, which has comparable order (Fig \ref{fig:timings}). However,  when we solve the complex valued differential equation $\dot{y} = \big(\frac{1}{2}- \frac{1}{2}i \big)y$, the running times are closer for the real and complex integrators.
 
  \begin{figure}
\begin{centering}
\includegraphics[width=5in]{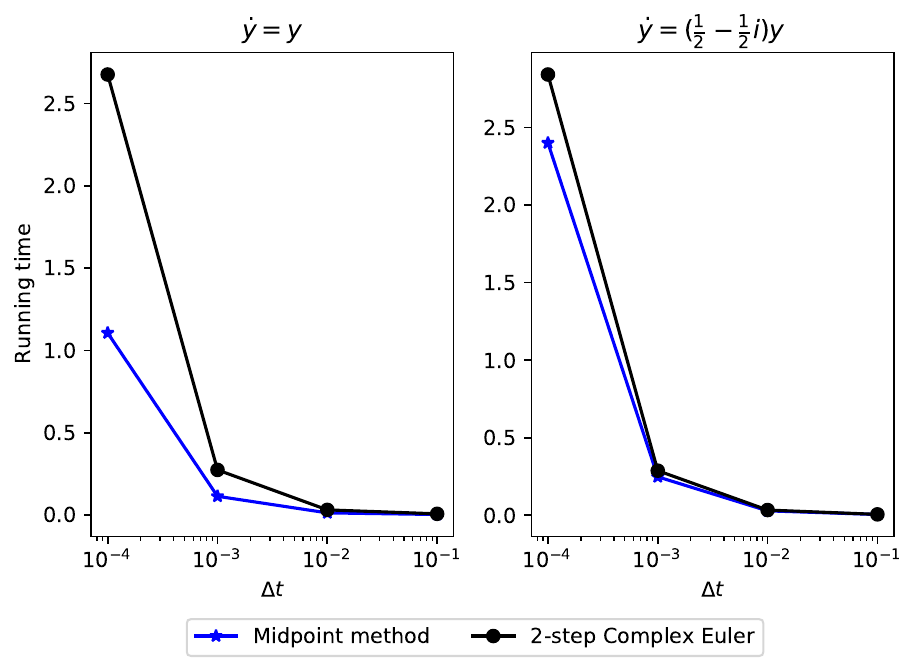}
\caption{Running times for the two step complex Euler method and the real explicit midpoint method for the differential equation $\dot{y} =y$ and  $\dot{y} = \big(\frac{1}{2}- \frac{1}{2}i \big)y$ solved from time 0 to 2 with various time steps ($\Delta t)$. We see the additional computational cost of complex time integrators for solving real differential equations. The additional cost is significantly less when solving complex valued differential equations.}
 \label{fig:timings}
\end{centering}

\end{figure}

For nonlinear systems, the dominant cost comes from function evaluations. Functions can be designed to minimize the computational cost, but the real equivalent of any complex-valued function can always be calculated using fewer operations. So, for  real valued differential equations, the complex time stepper will always be more expensive than an explicit Runge-Kutta method with the same number of function evaluations. Therefore, complex time steppers are more likely to be advantageous for  the numerical solution of complex differential equations. In the next two sections, we discuss unique advantages obtained by using complex integrators.

In Section \ref{sec:orderbarrier}, we demonstrate that complex time integrators require fewer function evaluations than the classical Runge-Kutta methods of the same order by circumventing the Runge-Kutta order barrier \cite{butcher2009fifth}. In that case, when the cost of additional function evaluations is more than the cost of complex operations, the complex integrators may be the less expensive option. In Section \ref{stabilitysection}, we see how complex time integrators allow for expanded stability regions for complex differential equations.

\section{Breaking the order barrier for explicit Runge-Kutta methods    }\label{sec:orderbarrier}

\subsection{What is the Runge-Kutta Order barrier?}

An $s$-stage Runge-Kutta method of order $p$ uses the free parameters in its $s$ stages to satisfy the order conditions necessary to acheive $p$-th order where $s \geq p$. Table \ref{tab:rk_conditions} lists the number of conditions needed to be satisfied at each order and the number of free parameters associated with the stages for a feasible Runge-Kutta method. Up to fourth order, one can construct explicit Runge-Kutta methods where $s = p$. For fifth order and beyond, the number of stages needs to be more than the order desired, i.e. $s > p$ \cite{butcher2009fifth}. The inability to construct Runge-Kutta methods with $ s =  p$ for order 5 or higher  is known as the order barrier for Runge-Kutta methods \cite{Butcher:2007}. In the following subsections, we show that it is possible to construct a 5-step, 5th order complex-time stepper for real-valued differential equations, and a 5-step complex-time-stepper for complex systems which satisfies the order conditions for 5th order accuracy up to error of $2.1 \times 10^{-11}$ 

\begin{table}[tbhp]
{\footnotesize
\caption{Conditions for various orders of accuracy and free parameters associated with Runge-Kutta methods with various stages}\label{tab:rk_conditions}
\begin{center}
\begin{tabular}{ |c|c|c|c|c|c|c|c|c|  } 
 \hline
 Order & 1 & 2 & 3 &4 & &5 &6 \\
   \hline
 Conditions & 1 & 2 & 4 &8 & &17 &37 \\
  \hline
 Stages & 1 & 2 & 3 & 4 &5 &6 &7  \\
  \hline
 Parameters& 1 & 3 & 6 &10 &15 &21 &28 \\ 
 \hline
\end{tabular}
\end{center}}
\end{table}

\subsection{Circumventing the Runge-Kutta order barrier for real-valued differential equations}

Let us take a specific look at the fifth order Runge-Kutta-5 method. A five stage method has 15 free parameters, which is insufficient to satisfy the 17 order conditions. Therefore, fifth order requires six stages and hence six function evaluations. What if we allowed some of the Runge-Kutta coefficients in a 5-stage Runge Kutta step to be complex variables and constrained the error to be purely imaginary beyond order 1?  Since we are working with real-valued differential equations, we can take the real part of the numerical solution at the end of each time step and achieve fifth order accuracy. This increases the number of free parameters from 15 to between 17 and 29 in order to solve the 17 order conditions! Table \ref{tab:rk_conditions_complex} describes the increase in free parameters when using Runge-Kutta methods with complex coefficients for real-valued problems.

\begin{table}[tbhp]
{\footnotesize
\caption{Increased free parameters when using  Runge-Kutta methods with complex coefficients for real valued differential equations}\label{tab:rk_conditions_complex}
\begin{center}
\begin{tabular}{ |c|c|c|c|c|c|c|c|c|  } 
 \hline

 Stages & 1 & 2 & 3 & 4 &5 &6 &7  \\
  \hline
 Original Parameters& 1 & 3 & 6 &10 &15 &21 &28 \\ 
 \hline
Parameters with complex coefficients & 1 & 5 & 11 &19 &29 &41 &55 \\ 
 \hline
\end{tabular}
\end{center}}
\end{table}
 
 \begin{figure}
\begin{centering}
\includegraphics[width=5in]{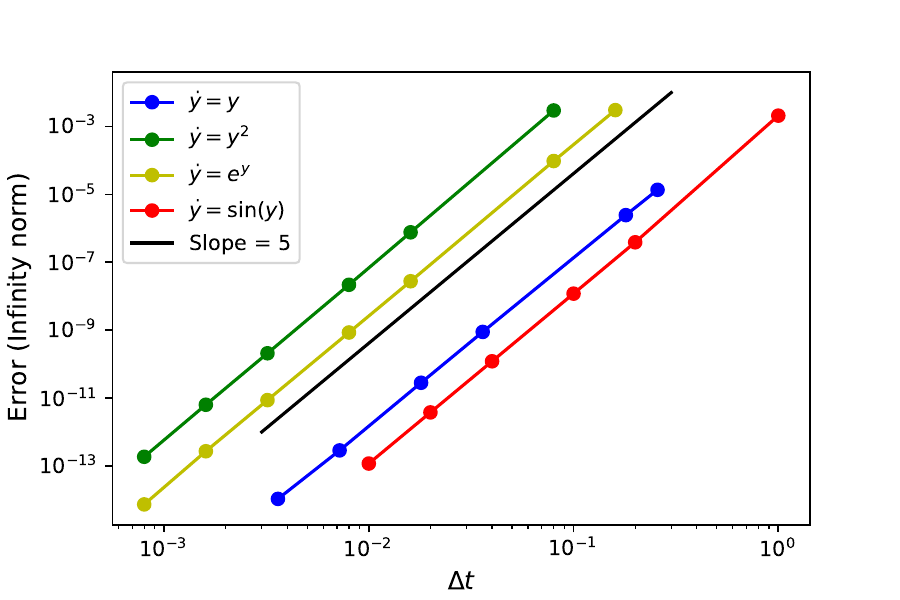}
\caption{The 5 stage 5th order Runge-Kutta method with complex Runge-Kutta coefficients has 5th order accuracy for real-valued differential equations. }
 \label{fig:order_barrier}
\end{centering}
\end{figure}

The following equations describe the complex Runge-Kutta-5 method for real-valued systems, 
\begin{align}
    \label{orderbarrier}
    y_n &= y_n\\
    k_{1} &= f(t_n, y_n)\\
    k_{2} &= f(t_n + c_2 \Delta t, y_n + \Delta t a_{21} k_{11})\\
    k_{3} &= f(t_n + c_3 \Delta t, y_n + \Delta t (a_{31} k_{1}+a_{32} k_{2}))\\
     k_{4} &= f(t_n + c_4 \Delta t, y_n + \Delta t (a_{41} k_{1}+a_{42} k_{2}+a_{43} k_{3}))\\
     k_{5} &= f(t_n + c_5 \Delta t, y_n + \Delta t (a_{51} k_{1}+a_{52} k_{2}+a_{53} k_{3}+a_{54} k_{4}))\\
    y_{n+1}  &= \textbf{Re}(y_n +b_{1}  \Delta t k_{1}+b_{2}  \Delta t k_{2}+b_{3}  \Delta t k_{3}+b_{4}  \Delta t k_{4}+b_{5}  \Delta t k_{5}).
\end{align}

With the constraints $c_i = \sum_j a_{ij}$ and $\sum_i b = 1$, these equations describe 5th-order 5-stage Runge-Kutta methods for real-valued differential equations with complex-valued coefficents. The coefficients for one such method are listed in Table \ref{tab:rk5_real} and fifth order accuracy is satisfied for the differential equations in Fig \ref{fig:order_barrier}.

\begin{table}[tbhp]
{\footnotesize
\caption{5 stage Runge-Kutta method for real differential equations.}\label{tab:rk5_real}
\begin{center}
\begin{tabular}{|c|c|} \hline
\bf Parameters & \bf Coefficients\\ \hline
$a_{21}$ & 0.4359927813681785+0.18820134969500546i \\
$a_{31}$ & 0.5984581874875472-0.6801332593573275i \\
$a_{32}$ & 0.09443736474929139+0.9536785997657906i \\
$a_{41}$ & -0.5318588311678385+0.06199640671232824i \\
$a_{42}$ & 0.7090327838155295+0.17964710178664897i \\
$a_{43}$ & 0.7502336256211084+0.014717632306291894i \\
$a_{51}$ & 0.11597306658216743+0.19224587759603343i \\
$a_{52}$ & -1.211955728302135+0.6697664876487938i \\
$a_{53}$ & 1.2481894547610273-1.0517638511367862i \\
$a_{54}$ & 1.1414853262483962+0.48897430346527126i \\
$b_1$ & 0.14051930946802596+0.047034144968353016i \\
$b_2$ & 0.5387707041084535+0.40236901283300025i \\
$b_3$ & 0.28423712936738976-0.23543136671378956i \\
$b_4$ & 0.06199686687229152-0.21051296375579337i \\ 
$b_5$ & -0.02552400981616073-0.003458827331770331i \\\hline
\end{tabular}
\end{center}
}
\end{table}

 Complex coefficients thus allow us to circumvent the Runge-Kutta order barrier and create a 5-stage fifth order method for real-valued differential equations. 
 \subsection{ Five stage, approximately fifth order complex Runge-Kutta methods}
 
In the previous subsection, we saw that complex coefficients can thus be used to break the order barrier for real-valued differential equations. Even larger benefits may occur at higher orders, where $s$ is much larger than $p$ for existing Runge-Kutta methods with real coefficients. However, since we take the real part of the solution at the end of every complex path, this approach cannot be used for complex-valued differential equations. As the additional computational cost of complex time integration is negligible for complex-valued systems, it is desirable to create a 5-stage Runge Kutta method where both the real and imaginary parts of the error are as close to zero as possible, approaching order 5. 

Real solutions to the order conditions are a subset of complex solutions and thus, complex solutions have the potential to satisfy the order conditions when there are no real solutions. We searched for complex solutions to the order conditions for fifth order accuracy using five stages. Root-finding and minimization algorithms from Python's Scipy package \cite{2020SciPy-NMeth} and Matlab's optimization toolbox \cite{MatlabOTB} were used in the numerical search for possible complex solutions. We found complex coefficients that satisfy the order conditions to a total error of $2.1\times 10^{-11}$. This means that for the differential equation $\dot{y} = f(t, y)$, the error behaves like $2.1\times 10^{-11} \Delta t^4 \dddot{f} + O(\Delta t^5)$. We even found real solutions that satisfy the order conditions to a total error of $1.8\times 10^{-9}$. The coefficients for both methods are listed in Table \ref{tab:rk5_complex_coeff}. We tested the complex five stage Runge-Kutta method on a few differential equations to obtain the convergence plot shown in Fig \ref{fig:order_barrier2}.

\begin{table}[tbhp]
{\footnotesize
\caption{5 stage Runge-Kutta method with complex and real coefficients that approximately satisfy the order conditions for fifth order accuracy. }\label{tab:rk5_complex_coeff}
\begin{center}
\begin{tabular}{|c|c||c|} \hline
\bf Parameters & \bf Complex Coefficients & \bf Real Coefficients \\ \hline
$a_{21}$  &  1.856587156265275e-07+1.5309457192095022e-07i &  5.254899676102671e-07 \\ 
$a_{31}$  &  355378.2918682022+744398.7276677284i  &  -282414.4914234111 \\ 
$a_{32}$  &  -355377.7953985455-744399.1156280392i  &  282415.0362283838 \\ 
$a_{41}$  &  10087.244864198223+2889.0099565661917i &  2300.659307961569 \\ 
$a_{42}$  &  -10086.873754015176-2889.502710365815i  &  -2300.39437640888 \\  
$a_{43}$  &  0.6299769187106239+0.4890885486059816i  &  0.355521993237099 \\  
$a_{51}$  &  16933.145111205715+9895.134727417835i &  -47221.11292217593 \\ 
$a_{52}$  &  -16932.764260866286-9895.630239734079i &  47221.41809024295 \\ 
$a_{53}$  &  0.6179505431419234+0.49914380654207474i &  -0.5826235568166092 \\ 
$a_{54}$  &  0.001199117424035724-0.003631490298717103i &  1.277455493703932 \\ 
$b_{1}$  &  -46564.847414291915+214551.5532581192i &  -51977.8184877715 \\ 
$b_{2}$  &  46565.24321098434-214551.70058574365i &  51978.11194824268 \\  
$b_{3}$  &  0.20881428641527866+0.0021225559323642816i  &  0.1667650923273279 \\
$b_{4}$  &  5.083449173489563-12.796017531317302i &  0.4161357937120537 \\  \hline
\end{tabular}
\end{center}
}
\end{table}

 \begin{figure}
\begin{centering}
\includegraphics[width=5in]{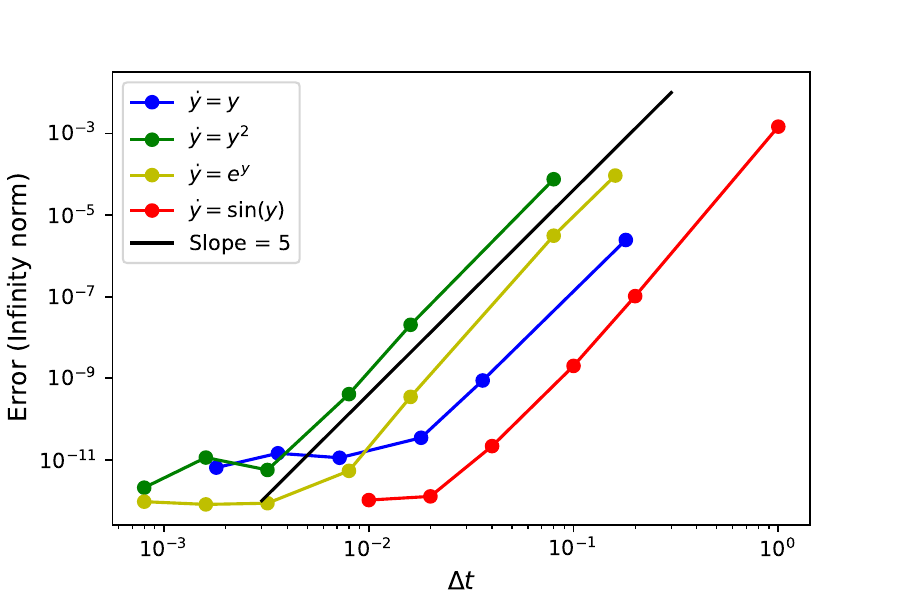}
\caption{The 5 stage Runge-Kutta method that satisfied the order conditions for fifth order accuracy  to an error of $2.1\times 10^{-11}$}
 \label{fig:order_barrier2}
\end{centering}
\end{figure}

Here, we have shown the existence of  five-stage fifth-order methods up to certain precision for both real and complex valued differential equations. To our knowledge, this has not been shown before. 

We tested our approximately fifth order complex Runge-Kutta method on the linear Schrodinger equation $u_t = i u_{xx}, x \in [0, 2\pi], t\in [0, 10]$ with periodic boundary conditions. The spatial differentiation is done using the the Fast Fourier Transform (FFT) with 100 modes. The exact solution is given by $u(x, t) = e^{i(x-t)}+e^{2i(x-2t)}$. In Table \ref{tab:schro_table}, we compare the average computational time and error ($L^1$ grid function norm) from ten simulations of the Schrodinger equation by our approximately fifth order complex five stage Runge-Kutta method (CRK5) to a traditional 5th order six stage Runge-Kutta method, Fehlberg's method with Formula 2 (RKF). 
\begin{table}[tbhp]
{\footnotesize
\caption{Computational time (seconds) and error in the simulation of the Schrodinger equation by the complex Runge-Kutta method (CRK5) and the Fehlberg method (RKF)}\label{tab:schro_table}
\begin{center}
\begin{tabular}{|c|c|c||c|c||} \hline
\bf Step size & \bf Time  (CRK5) & \bf Time (RKF5) & \bf Error (CRK5) & \bf Error (RKF) \\ \hline
$0.0002$  &  6.41 &  7.52 &  2.44e-08& 1.01e-13 \\  \hline
$0.0001$   &  13.01 &  15.53 &  1.28e-08& 1.24e-13\\ \hline
$0.00005$   &  25.91 &  30.29&  9.99e-09& 1.64e-13 \\ \hline
\end{tabular}
\end{center}
}
\end{table}

We can clearly see from  Table \ref{tab:schro_table} that although the complex  5-stage method is limited by its precision, it is less computationally expensive when compared to a traditional fifth order Runge-Kutta method. We hope this example motivates searches for complex solutions to order conditions especially in the case of complex-valued differential equations.

\section{Expanded regions of stability with complex time-steps} \label{stabilitysection}

In the previous sections, we designed paths in the complex plane to achieve higher order. In this section, we leverage complex paths in an alternative way and expand the time steppers's region of absolute stability. First, we discuss the linear stability regions of the complex integrators developed in Sec. \ref{subsec:linearFEpaths}. Next, we show how these regions can be expanded by designing the complex steps to optimize a desired stability polynomial. 
Expanded stability regions have been studied in a similar manner previously in the context of Runge-Kutta methods where the order of accuracy for the Runge-Kutta method is sacrificed to obtain a larger stability region \cite{riha1972optimal, lawson1966order,abdulle2000roots, bogatyrev2004effective, ketcheson2013optimal, parsani2013optimized, kubatko2014optimal}. In most previous works, a new Runge-Kutta integrator (different integration method) was derived that had a stability function/polynomial corresponding to the desired stability region. Here, we keep the same integration method and traverse a path in the complex plane to obtain the desired stability region (different integration domain). We illustrate this approach using the forward Euler method. We also show that complex stepping may enable even larger regions designed for specific problems such as complex valued differential equations in quantum systems, where stability currently limits computational efficiency.

Consider the linear equation $\dot{y} = \lambda y$ studied in Sec. \ref{subsec:linearFEpaths}. Most integrators, including those discussed here, convert the linear equation into the difference equation $y(t+\Delta t) = \Phi(\lambda \Delta t)y(t)$ where $\Phi(z)$ is the stability function of the integrator \cite{leveque2007finite}. The region $z \in \mathbb{C}$ where $|\Phi(z)|\leq 1$ is the region of absolute stability. For a complex $n$-step $n$th order method, the stability function is given by the first $n +1$ terms in the Taylor  series of $e^{\Delta t}$. For the 3 step 3rd order method, the stability function is

\begin{align}
\Phi(z) = 1 + z + \frac{z^2}{2} + \frac{z^3}{3!}.
\end{align}

The regions of absolute stability for the 1 step 1st order, 2 step 2nd order and 3 step 3rd order methods (Fig \ref{4_1}) are the same as those for the n-stage Runge-Kutta because they have the same stability functions \cite{leveque2007finite}. Note that $\Delta t$ in all these methods refers to the step size separating the beginning and end points of the complex path and hence a real time step. Since the 1-step, 2-step and 3-step methods require a different number of steps and hence a different number of function evaluations to take a net step of size $\Delta t$, a standard metric for comparison among methods is the effective step size $\Delta t_{eff}  = \Delta t/n$ where $n$ is the number of steps in the complex integrator case or the number of stages in the case of Runge Kutta methods \cite{ketcheson2013optimal}.

 \begin{figure}
\begin{centering}
\includegraphics[width=5in]{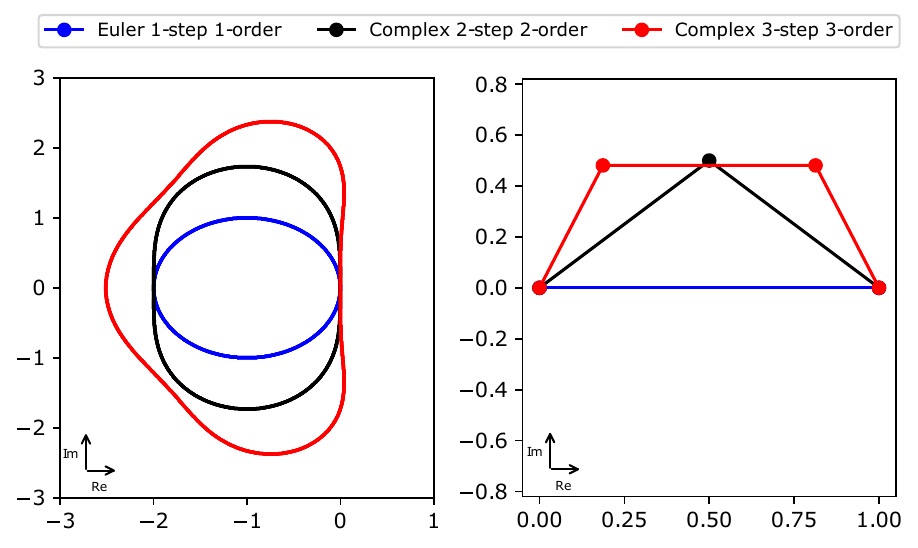}
\caption{The 1-step first order, 2-step 2nd order and 3-step 3rd order complex Euler methods share the same stability regions as the 1,2 and 3 stage Runge-Kutta methods (left). Corresponding complex time-stepping paths are given on right.}
 \label{4_1}
\end{centering}
\end{figure}

We would like to expand the regions in Fig \ref{4_1} by optimizing the stability function. If we relax our requirement for a 3-step integrator to be 3rd order accurate, and require only 2nd order accuracy, its stability function would be given by

\begin{align}
\Phi(z) = 1 + z + \frac{z^2}{2} + k z^3.
\end{align}

The coefficient $k$ in the above expression is the same as the coefficient in front of the $\lambda^3 \Delta t^3$ in equation \eqref{eq:linearcoeff}: $w_1w_2w_3$. We can choose the complex steps $w_1, w_2$ and $w_3$ such that they satisfy the first and second order conditions ($w_1 + w_2 + w_3 = 1, w_1 w_2 +w_2w_3 +w_1w_3 = 1/2$)and make $w_1 w_2 w_3$ equal to our desired value of $k$. The optimal $k$ value is the one that would allow us to take the largest stable time step. For problems with purely real and imaginary eigenvalues, the optimal $k$ values and the associated optimal stability functions/polynomials have been long-studied \cite{riha1972optimal,bogatyrev2004effective,abdulle2000roots}. In the past decade, Ketcheson and co-authors \cite{parsani2013optimized, kubatko2014optimal, ketcheson2013optimal} have explored optimal stability functions in the context of $\dot{y} = L y$ where $L$ is a matrix with quite arbitrary spectra. In \cite{ketcheson2013optimal}, Ketcheson and Ahmadia  describe an optimization algorithm to obtain the optimal stability polynomial that allows the maximal time step depending on the spectra of $L$. This algorithm (RKOpt \cite{ketcheson2020rk}) is publicly available through Nodepy \cite{ketcheson2020nodepy}.

By choosing the value $k$ that results in the largest permissible time step along the negative real eigenvalue line, we achieve a 3-step 2nd order method (blue) with expanded stability along the negative real axis compared to the 3-step 3rd order method (red Fig \ref{4_2}). Requiring only 1st order accuracy for the 3-step method can nearly triple the size of the optimal stability region along the negative real axis (gold) and surprisingly requires only real time-steps (Fig \ref{4_2}). 
  \begin{figure}
\begin{centering}
\includegraphics[width=5in]{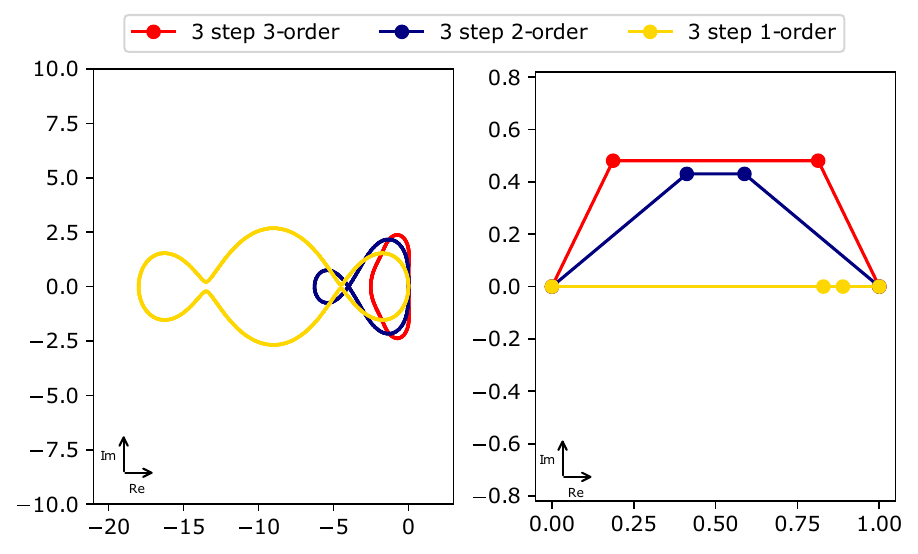}
\caption{By sacrificing accuracy, one can expand the stability region (left) using complex time steps, as demonstrated by these optimized 3-step complex Euler paths (right) with 3rd, 2nd, and 1st order accuracy.}
\label{4_2}
\end{centering}
\end{figure}

 To our knowledge, our work is the first to use complex time steps to create stability polynomials, enabling us to expand the stability region. Optimal stability polynomials with complex coefficients may have larger stability regions than those with purely real coefficients. Through numerical experimentation, we have found that complex-valued coefficients (in the stability polynomial) often result in stability regions that are not symmetric about the real axis. An example is shown in Fig \ref{4_3equation}. Asymmetric stability regions will likely be sub-optimal for real-valued problems, since the eigenvalues of real-valued problems appear as complex conjugate pair symmetric about the real axis and asymmetric regions cannot capture both pairs. This leads us to hypothesize that the optimal stability polynomials for real-valued problems likely have purely real coefficients.

However, for complex-valued problems which have asymmetric eigenvalue spectra, complex coefficients produce larger stability regions. Complex-valued problems often pop up in fields like quantum mechanics. In these settings, complex time steps require no additional cost and can often reduce computational effort by expanding stability regions. For example, the complex-valued Kohn-Sham equations which describe ultra-fast electron dynamics are known to require extremely small time steps for stability when using standard explicit Runge-Kutta methods \cite{kang2019pushing}. The search for efficient time integration techniques for complex-valued quantum systems that allow for large stable time steps, is an active field of research \cite{an2020quantum, an2022parallel, kononov2022electron}. In \cite{kononov2022electron}, a number of numerical methods are reviewed for the simulation of electron dynamics using the complex-valued Kohn-Sham equations and since implicit methods are far more expensive, they express a specific need for explicit methods with expanded stability regions.

Now, we demonstrate the expanded stability benefits of using complex coefficients by considering a classical complex-valued problem, the linear Schrodinger equation, whose eigenvalues lie on the negative imaginary axis. We wish to find the 2-stage/step first order integrator for the Schrodinger equation that allows for the maximal stable time step, i.e, we wish to find the $k$ in the stability polynomial below that covers the maximal length from the origin on the negative imaginary axis. The stability polynomial for a general 2-stage/step first order integrator is given by \ref{eq:schro_stab}. We are looking to find the optimal $k$ value that allows for the maximal stable time step.

\begin{align}
\label{eq:schro_stab}
    \Phi(z) = 1 + z + k z^2 
\end{align}

The optimal stability polynomial with purely real coefficients can be shown to be 

\begin{align}
    \Phi(z) = 1 + z +  z^2 
\end{align}

The optimal stability polynomial with complex coefficients can be shown to be 

\begin{align}
    \Phi(z) = 1 + z + \bigg(\frac{1}{2}-\frac{1}{2}i\bigg)z^2 
\end{align}

We see in Fig \ref{4_4equation} that the stability region using complex coefficients covers double the length on the negative imaginary axis as the one obtained using purely real coefficients. We show through this simple example complex steps/coefficients in a time integrator can expand stability regions beyond real equivalents and allow for larger time steps.

  \begin{figure}
\begin{centering}
\includegraphics[width=5in]{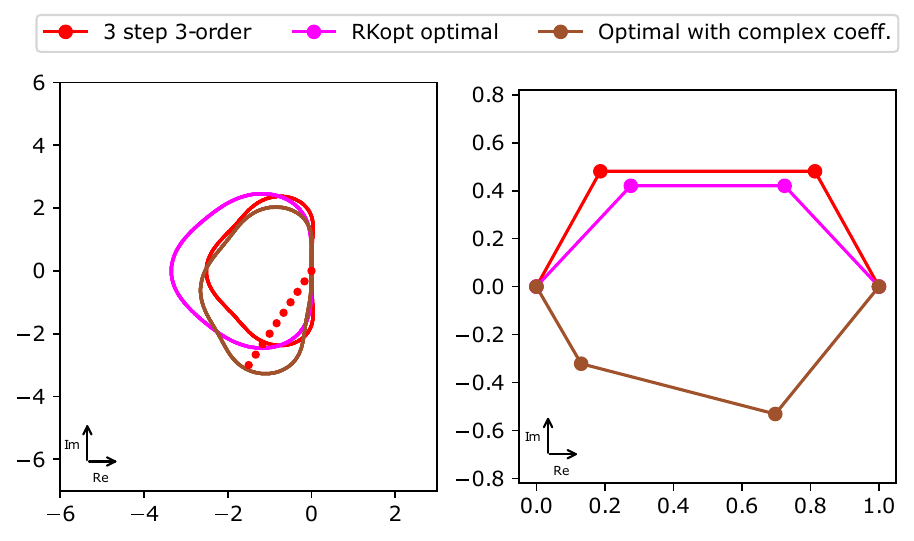}
\caption{ Complex coefficients in optimal stability polynomials often result in unsymmetrical stability regions. So, while the 3-step 2nd order complex integrator allows larger time steps along the red eigenvalues, the 3 stage 2nd order real integrator is still optimal if one has to take timesteps along both the red eigenvalues and its conjugate (as is the case for a real system). }
\label{4_3equation}
\end{centering}

\end{figure}

  \begin{figure}
\begin{centering}
\includegraphics[width=5in]{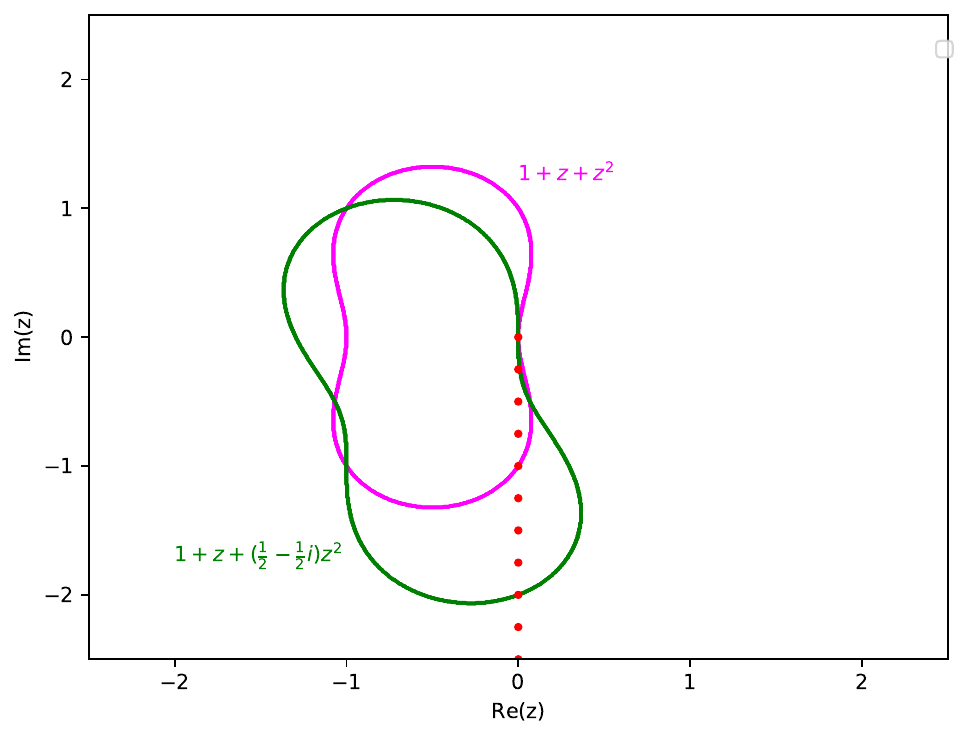}
\caption{ Constructing optimal stability polynomials using complex coefficients enables us to get larger stability regions than those using purely real coefficients. The optimal 2-stage/step 1st order stability polynomials for the linear Schrodinger equation using purely real and complex coefficients are shown here in magenta and green respectively. The possible timesteps for the linear Schrodinger equation along its eigenvalues are shown in red.}
\label{4_4equation}
\end{centering}

\end{figure}

\section{Discussion}
This paper is an exploration of the benefits and trade offs of leaving the real line to perform numerical  integration in the complex plane. Our goal was to explain the process of developing a complex time-stepper and its potential benefits in an accessible manner. We demonstrate how to systematically derive paths in the complex plane that increase the order of accuracy of any Runge-Kutta type integrator for general nonlinear differential equations. This process can be followed to similarly derive complex paths with other desired properties as well.

In Section \ref{subsec:nonlinearde}, we show that it is possible to get higher order accuracy by  taking complex time steps with a single stage method like the Euler method. An unmentioned advantage of this is that we only need to store one stage at every step. For example, third order Runge Kutta methods usually require 3 intermediate stages (3N registers) to stored at every step while the 3-step 3rd order Euler method only needs one stage of size 2N (because of the addition N imaginary values). Storage may not be much of an issue for simple systems but it becomes quite relevant in various practical settings like in \cite{KENNEDY2000177} where partial differential equations like the Navier Stokes equations are solved in 3 dimensions at high resolution. There are already low-storage algorithms for Runge-Kutta  schemes that use fewer registers \cite{carpenter1994fourth,KENNEDY2000177, ketcheson2008highly} but complex time integrators can provide an additional axis to create low-storage schemes with desired properties.

It is also possible to derive higher-order paths for specific differential equations. For example, the 2-step Forward Euler path with steps $1 -\frac{1}{\sqrt{2}} i$ and $\frac{1}{\sqrt{2}} i$ is 3rd order accurate for the differential equation $\dot{y} = - y^2$ with real initial conditions. There has been recent work\cite{guo2022personalized} using neural networks to obtain these uniquely accurate integrators for specific problems. Complex steps and coefficients in this context may reveal even better specialized integrators.

Complex time steppers enable us to circumvent the Runge-Kutta order barrier\cite{butcher2009fifth}. For real-valued differential equations, complex time steps and complex Runge-Kutta coefficients allow us to  increase the number of free parameters available to satisfy the order conditions and use fewer function evaluations to achieve the same order of  accuracy. Since the number of order conditions grow rapidly with higher order, we believe the best benefits lie at higher order. For example,  we demonstrated a Runge-Kutta-5 method with 5 function evaluations instead of the usual 6. However, an 8th order Runge-Kutta method traditionally requires least 11 function evaluations to obtain 8th order of accuracy. Complex timesteppers could reduce the number of evaluations down to 8. 

The disadvantage of complex timesteppers is the additional computational cost arising from complex operations. However, this cost is often negligible for complex-valued differential equations. Although there is no exact solution, we searched for complex parameters that satisfied the order conditions for 5th-order accuracy with 5 function evaluations. Since solutions to the order conditions with real parameters were a subset of solutions with complex parameters, we reasoned the complex parameters would give us more freedom to satisfy the order conditions. We were able to find coefficients that satisfied the order conditions to $2.1\times10^{-11}$ which may be very useful from a practical standpoint.

Lastly, we show that complex time stepping can be used to expand stability regions, allowing larger, stable time steps. Expanded stability regions have been designed using Runge-Kutta methods \cite{riha1972optimal, ketcheson2013optimal}, but we show that complex coefficients in the optimal stability polynomial allow for even larger stability regions, particularly for complex-valued systems. The fact that you can increase the stability region by changing the integration domain rather than the integration method may also be appealing in scenarios where particular integrators with specific properties (such as nonlinear stability) for specific differential equations are desired. Although we only explored linear stability analysis, it is possible that complex time integrators can similarly improve desired nonlinear stability properties (like Strong Stability Preserving properties \cite{gottlieb2001strong,gottlieb2009high}).

We intend this paper to be introduction to and initiation of the search for better time integrators that take advantage of paths in the complex plane. We have mostly focused on creating integrators with increased accuracy and stability, but there are many other desirable properties including energy and momentum preservation. The complex plane offers an extra dimension on which any integrator can be improved, opening up a new realm of possibilities when designing time stepping methods.

\section{Acknowledgements}

We would like to thank Emil Constantinescu, David Ketcheson, Alvin Bayliss and David Chopp for extensive feedback and valuable suggestions.

\bibliographystyle{unsrt}

\bibliography{references}
\end{document}